\RequirePackage{fix-cm}
\documentclass[smallextended]{svjour3}       
\smartqed  
\usepackage{graphicx}
\usepackage[colorlinks]{hyperref}
\usepackage{mathrsfs,enumitem,amsmath,amssymb}
\usepackage{tikz,tkz-euclide}
\usepackage{tikz-qtree}

\usetikzlibrary{calc,hobby}
\usetikzlibrary{decorations.pathmorphing,patterns,decorations,shapes,arrows, intersections,matrix,fit,calc,trees,positioning,arrows,chains, shapes.geometric,shapes,angles,quotes,fit,math}
\usetikzlibrary{arrows,positioning}
\usetikzlibrary{decorations.markings}

\def\ttitle{Effect of density dependence on coinfection dynamics: part 2}

\tikzset{
  big arrow/.style={
    decoration={markings,mark=at position 1 with {\arrow[scale=2,#1]{>}}},
    postaction={decorate},
    shorten >=0.4pt}}

\begin{document}

\title{\ttitle
}



\author{Jonathan Andersson\and Samia Ghersheen \and Vladimir Kozlov \and Vladimir~G.~Tkachev* \and Uno~Wennergren}

\authorrunning{J. Andersson, S. Ghersheen, V. Kozlov, V. Tkachev, U. Wennergren} 

\institute{J. Andersson \at
              Department of Mathematics, Link\"oping University \\
              \email{jonathan.andersson@liu.se}           
           \and
           S. Ghersheen\at
            Department of Mathematics, Link\"oping University \\
              \email{samia.ghersheen@liu.se}
              \and
            V. Kozlov\at
            Department of Mathematics, Link\"oping University \\
              \email{vladimir.kozlov@liu.se}
              \and
           V. Tkachev\at
            Department of Mathematics, Link\"oping University \\
              \email{vladimir.tkatjev@liu.se}
              \and
              U. Wennergren\at
Department of Physics, Chemistry, and Biology, Link\"oping University\\
\email{uno.wennergren@liu.se}
    }

\date{Received: date / Accepted: date}

\maketitle
	\begin{abstract}
In this paper we continue the stability analysis of the model for coinfection with density dependent susceptible population introduced in \cite{part1}. We consider the remaining parameter values
left out from \cite{part1}. We look for coexistence equilibrium points, their stability and dependence on the carrying capacity $K$. Two sets of parameter value are determined,
each giving rise to different scenarios for the equilibrium branch parametrized by $K$. In both scenarios the branch includes coexistence points
implying that both coinfection and single infection of both diseases can exist together in a stable state. There  are  no  simple  explicit  expression  for  these  equilibrium  points  and  we  will require a more delicate analysis of these points with a new bifurcation technique adapted to such epidemic related problems.
The first scenario is described by the branch of stable equilibrium points which includes a section of coexistence points starting at a bifurcation equilibrium point with zero second single infections and finishing at another bifurcation point with zero first single infections. In the second scenario the branch also includes a section of coexistence equilibrium points with the same type of starting point but the branch stays inside the positive cone after this. The coexistence equilibrium points are stable at the start of the section. It stays stable as long as the product of $K$ and the rate $\bar \gamma$ of coinfection resulting from two single infections is small but, after this it can reach a Hopf bifurcation and periodic orbits will appear.
\end{abstract}	
\keywords{SIR model, coinfection, carrying capacity, global stability}
\maketitle
	
	\section{Introduction} In this paper we continue on the work of \cite{part1} where we studied the equilibrium dynamics for a continuous compartmental model of two infectious diseases with the ability to co-infect individuals. In the model we assume that only the susceptibles can give birth and that the reproductive rate depends on the density of the susceptibles. This dependence is modelled with a parameter $K>0$ which represent the carrying capacity of the population. Recall that by an \textit{(equilibrium) branch} we understand any continuous in $K\ge0$ family of equilibrium points of a dynamic system which are locally stable for all but finitely many threshold values of $K$.

In \cite{part1} it was discovered that for a certain set of parameters excluding $K$ there exists an equilibrium branch with respect to $K$ of locally stable equilibrium. For this branch all of the equilibrium points where expressed explicitly and $K$ for which a compartment changed from being zero to non-zero or vice versa where pointed out.

In this paper we will show the same holds for the rest of the parametric choices. 	The main difficulty compared to \cite{part1} is that for our parameters the equilibrium branch consists of coexistence equilibrium where single infection of each disease and coinfection both occurs. There are no simple explicit expression for these equilibrium points and we will require a more delicate analysis of these points with a new bifurcation technique adapted to such epidemic related problems.


\subsection{The model}

As in \cite{part1}, we assume that the single infection cannot be transmitted by the contact with a coinfected person. This process gives rise to the  model:
\begin{equation*}\label{submodel2}
\begin{split}
S' &=(r(1-\frac{S}{K})-\alpha_1I_1-\alpha_2I_2-\alpha_3I_{12})S,\\
I_1' &=(\alpha_1S - \eta_1I_{12}-\gamma_1I_2 - \mu_1)I_1,\\
I_2' &=(\alpha_2S - \eta_2I_{12}-\gamma_2I_1- \mu_2)I_2,\\
I_{12}' &=(\alpha_3S+ \eta_1I_1+\eta_2I_2-\mu_3)I_{12}+\overline{\gamma}I_1I_2, \\
R' &=\rho_1 I_1+\rho_2I_2+\rho_3 I_{12}-\mu_4' R,
\end{split}
\end{equation*}
where we use the following notation:
\begin{itemize}
\item[$\bullet$] $S$ represents the susceptible class,
\item[$\bullet$] $I_1$ and $I_2$ are the infected classes from strain 1 and strain 2 respectively,
\item[$\bullet$] $I_{12} $ represents the co-infected class,
\item[$\bullet$]
$R$ represents the recovered class.
\end{itemize}

Following \cite{Allen,Bremermann,Zhou}, we assume limited population growth by making the per capita reproduction rate depend on the density of population. We also consider the recovery of each infected class (see the last equation in \eqref{submodel2}). The fundamental  parameters of the system are:
\begin{itemize}
\item[$\bullet$] $r=b-d_0$ is  the intrinsic rate of natural increase, where $b$ is the birthrate and $d_0$ is the death rate of $S$-class,
\item[$\bullet$]
$K $ is the carrying capacity (see also the next section),
\item[$\bullet$]
$\rho_i$ is the recovery rate from each infected class ($i=1,2,3$),
\item[$\bullet$] $d_i $ is the death rate of each class,  $( i=1,2,3,4)$, where $d_3$ and $d_4$ correspond $I_{12}$ and $R$ respectively,
\item[$\bullet$]
$\mu_i=\rho_i+d_i, i=1,2,3.$

\item[$\bullet$]
$\alpha_1$, $\alpha_2$, $\alpha_3$  are the rates of transmission of strain 1, strain 2 and both strains (in the case of coinfection),
\item[$\bullet$]
$\gamma_i$ is the rate at which infected with one strain get infected with  the other strain and move to a coinfected class ($i=1,2$),

\item[$\bullet$]     $\eta_i$ is the rate at which infected from one strain getting  infection from a co-infected class $( i=1,2)$;
\end{itemize}

We only consider the case when the reproduction rate of susceptibles is not less than their death rate since we know that the population will go extinct in that case.
The system is considered under the natural initial conditions
$S(0)>0$, $I_1(0)\geq0$, $I_2(0)\geq0$, $I_{12}(0)\geq0$ and by $N=S+I_1+I_2+I_{12}+R.$ we denote the total population.


Since the variable $R$ is not present in the first four equations, without loss of generality, we may consider only the first four equations of system \eqref{submodel2}.
It is convenient to introduce the notation
\begin{equation}\label{sigma_i ordering}
\sigma_i:=\frac{\mu_i}{\alpha_i}, \qquad 1\le i\le 3.
\end{equation}
And similarly to the first part \cite{part1} we make the following assumption
\begin{equation}\label{assum}
\sigma_1< \sigma_2 < \sigma_3.
\end{equation}
We shall also assume that \eqref{submodel2} satisfies the non-degenerate condition
\begin{equation}\label{deltamu}
\Delta_\alpha=\left|
              \begin{array}{cc}
                \eta_1 & \alpha_1 \\
                \eta_2 & \alpha_2 \\
              \end{array}
            \right|=\eta_1\alpha_2-\eta_2\alpha_1\ne0.
\end{equation}
This condition have a natural biological explanation: the virus strains 1 and 2 have different (co)infections rates.
We use the notation
$$
\bar \gamma= \gamma_1+\gamma_2,\;\;\gamma=(\gamma_1,\gamma_2),
$$
and
\begin{eqnarray}
A_1&= \frac{\alpha_1\alpha_3}{r}(\sigma_3-\sigma_1),& \eta_1^*:= \frac{\eta_1}{A_1}\nonumber\label{A1}\\
 A_2&= \frac{\alpha_2\alpha_3}{r}(\sigma_3-\sigma_2), & \eta_2^*:=\frac{\eta_2}{A_2}\nonumber\label{A2}\\
 A_3&= \frac{\alpha_1\alpha_2}{r}(\sigma_2-\sigma_1), & \gamma^*:=\frac{\gamma_1}{A_3}\nonumber\label{A3}.
 \end{eqnarray}
Notice that by \eqref{assum} one has $A_1,A_2,A_3>0$.
 We have
  \begin{equation}\label{etaA123}
  \alpha_2A_1=\alpha_3A_3+
  \alpha_1A_2.
  \end{equation}
The  determinants $\Delta_\alpha$ and $\Delta_\mu=\eta_1\mu_2-\eta_2\mu_1$ are related to each other by
\begin{equation*}\label{deltamu1}
\Delta_\mu=\frac{\eta_1r}{\alpha_1}A_3+\sigma_1\Delta_\alpha=
\frac{\eta_2r}{\alpha_2}A_3+\sigma_2\Delta_\alpha,
\end{equation*}
hence $A_3>0$ implies
\begin{equation}\label{deltamu2}
\Delta_\mu>\sigma_1\Delta_\alpha \qquad \Delta_\mu>\sigma_2\Delta_\alpha.
\end{equation}
This implies an inequality which will be useful in the further analysis:
\begin{equation}\label{deltamu12}
\sigma_2(\Delta_\alpha+\gamma_2\alpha_3)<\Delta_\mu+\gamma_2\mu_3.
\end{equation}
 We shall also make use of the following relations:
 \begin{equation}\label{etaA13}
\begin{split}\eta_1^* -\eta_2^*&< \eta_1\frac{\alpha_2}{\alpha_1A_2}-\eta_2^*
= \frac{\Delta_\alpha}{\alpha_1A_2}.
\end{split}
\end{equation}
A consequence of \eqref{etaA13} and \eqref{deltamu2} is that for $\eta_1^*>\eta_2^*$ we have $\Delta_\alpha,\;\Delta_\mu>0$.
On the other hand, one has
\begin{eqnarray}\label{etaAetaA}
\eta_1^*- \eta_2^* =
\frac{(\Delta_\alpha\sigma_3-\Delta_\mu)\alpha_3}{rA_1A_2}
\end{eqnarray}

\subsection{The main result}
It is 
elementary to see that except for the trivial equilibrium state $G_1=(0,0,0,0)$ and the disease free equilibrium $G_2=(K, 0,0,0)$, there exist only $6$ possible \textit{types of equilibrium points} $G_3$, $G_4$, $G_5$, $G_6$, $G_7$, $G_8$ determined by their non-zero compartments (see Table~\ref{tab1} and Proposition~\ref{pro:equil} below for explicit representations).

\begin{table}[h]
  \centering
\begin{tabular}{c|c|c|c|c}
\quad Type \quad &\quad $S$&$I_1$&$I_2$&$I_{12}$\\\hline
$G_2$& $\star$ & $0$& $0$& $0$\\\hline
$G_3$& $\star$ & $\star$& $0$& $0$\\\hline
$G_4$& $\star$ & $0$& $\star$& $0$\\\hline
$G_5$& $\star$ & $0$& $0$& $\star$\\\hline
$G_6$& $\star$ & $\star$& $0$& $\star$\\\hline
$G_7$& $\star$ & $0$& $\star$& $\star$\\\hline
$G_8$& $\star$ & $\star$& $\star$& $\star$\\
\end{tabular}
  \caption{The  types of equilibrium states of \eqref{submodel2}, where $\star$ denotes a non-zero coordinate}\label{tab1}
\end{table}

More precisely, the equilibrium points $G_3$, $G_4$, $G_5$ have two non-zero components and represent points where only one of the diseases are present or where the diseases only exist together as coinfection. At the points $G_6,G_7$ one of the diseases are only present in coinfected individuals while the other disease also occurs as single infections. The point $G_8$ is the coexistence equilibrium were both types of single infection is present as well as coinfection. Our main results extends the results of \cite{SKTW18a} on the case of small values of $\gamma_i$. More precisely, we have only four possible scenarios of developing of a locally stable equilibrium point as a continuous function of increasing carrying capacity $K$:

\begin{theorem}\label{th:main}
Let all parameters $\alpha_i,\mu_i,\eta_i,\gamma_i$ of \eqref{submodel2} be fixed with $\bar\gamma$ sufficiently small. Then one has exactly one locally stable nonnegative equilibrium point depending on $K>0$. Furthermore, changing the carrying capacity $K$ from zero to infinity, the type of this locally stable equilibrium point changes according to one of the following alternative scenarios:
\begin{enumerate}[label=(\roman*),itemsep=0.1ex,leftmargin=0.8cm]
        \item \label{it1}$G_2\rightarrow G_3$;
        \item \label{it4}$G_2\rightarrow G_3\rightarrow G_6\rightarrow G_5$;
            \item \label{it2} $G_2 \rightarrow G_3\rightarrow G_6\rightarrow G_8\rightarrow G_7\rightarrow G_5$;
        \item \label{it3}$G_2\rightarrow G_3\rightarrow G_6\rightarrow G_8$.
    \end{enumerate}
\end{theorem}

The first two scenarios are considered in our paper \cite{part1}.
In this paper we consider  the remained two scenarios, (iii) and (vi). These cases require a more nontrivial bifurcation analysis with application
of methods similar to the principle of the exchange of stability developed in
\cite{kielhofer2012introduction}, see also \cite{liu1994criterion} and \cite{part1} for recent applications in population analysis. In our
context, this require a delicate analysis of the inner equilibrium state $G_8$, as
well as a new bifurcation technique.

\section{Equilibrium points}\label{equipoint}

We note that the last equation in \eqref{submodel2} can be solved explicitly with respect to $R$:
\begin{equation*}
    R(t)=e^{-\mu'_4t}R(0)+\int_0^t e^{\mu'_4(\tau-t)}(\rho_1I_1+\rho_2I_2+\rho_3I_{12})(\tau) d\tau
\end{equation*}
therefore it suffices to study the dynamics of the first four equations in \eqref{submodel2}. If $I_1,I_2$ and $I_{12}$ have limits $\hat I_1$, $\hat I_2$ and $\hat I_{12}$ respectively as $t\rightarrow \infty$ then $R$ will have the limit
\begin{equation*}
    \hat R= \frac{\rho_1\hat I_1+\rho_2\hat I_2+ \rho_3 \hat I_{12}}{\mu_4}
\end{equation*}

Let us turn to the first four equations in \eqref{submodel2}.
The equilibrium points satisfy the following system
\begin{equation}\label{Equilib}
\begin{split}
(b(1-\frac{S}{K})-\alpha_1I_1-\alpha_2I_2-\alpha_3I_{12}-\mu_0)S=0,\\
(\alpha_1S - \eta_1I_{12}-\gamma_1I_2 - \mu_1)I_1=0,\\
(\alpha_2S - \eta_2I_{12}-\gamma_2I_1- \mu_2)I_2=0,\\
(\alpha_3S+ \eta_1I_1+\eta_2I_2-\mu_3)I_{12}+\overline{\gamma}I_1I_2=0.
\end{split}
\end{equation}
and in \cite{part1} we had the following proposition
\begin{proposition}\label{pro:equil}
Except for the trivial equilibrium $G_1=(0,0,0,0)$ and the disease free equilibrium $G_2=(K, 0,0,0)$ there exist only the following  equilibrium states:
\begin{eqnarray*}\label{subeqpt3}
G_3&=&\left(\sigma_1, \frac{r}{K\alpha_1}(K-\sigma_1), 0,0\right),\nonumber\\
\label{G4coor}
G_4&=&(\sigma_2,0,\frac{r}{K\alpha_2}(K-\sigma_2), 0),\nonumber\\
\label{subeqpt5}
G_5&=&(\sigma_3,0,0,\frac{r}{K\alpha_3}(K-\sigma_3)),\nonumber\\
\label{subeqpt6}\nonumber
G_6&=&(S^*,\frac{\alpha_3}{\eta_1}(\sigma_3-S^*),0,\frac{\alpha_1}{\eta_1}(S^*-\sigma_1)), \quad\text{where}\quad S^*=K(1-\frac{1}{\eta_1^*}), \,\,
\\
G_7&=&(S^*,0,\frac{\alpha_3}{\eta_2}(\sigma_3-S^*),\frac{\alpha_2}{\eta_2}(S^*-\sigma_2)), \quad\text{where}\quad
S^*=K(1-\frac{1}{\eta_2^*}),\,\,
\label{subeqpt7}\\
G_8&=&(S^*,I_1^*,I_2^*,I_{12}^*).\nonumber
\end{eqnarray*}
\end{proposition}

All required information about equilibrium points $G_1$ -- $ G_6$ can be found in \cite{part1}. Here we take only points $G_8$ and $G_7$. To highlight the dependence of the equilibrium points on $K$ we will write sometimes $G_j(K)$.

\subsection{Coexistence equilibria}
The coordinates of coexistence equilibrium points satisfy
\begin{equation}\label{Equilib1}
\left\{
\begin{split}
r(1-\frac{S}{K})-\alpha_1I_1-\alpha_2I_2-\alpha_3I_{12}=0,\\
\alpha_1S - \eta_1I_{12}-\gamma_1I_2 - \mu_1=0,\\
\alpha_2S - \eta_2I_{12}-\gamma_2I_1- \mu_2=0,\\
\alpha_3S+\eta_1I_1+\eta_2I_2-\mu_3+\frac{\overline{\gamma}I_1I_2}{I_{12}}=0
\end{split}
\right.
\end{equation}
Furthermore, as it is shown in \cite{part1}, Sect. 3.2, the $S$ coordinate of an inner equilibrium point (coexistence equilibrium) satisfies
$P(S)=0$
where
\begin{equation*}\label{Peq2}
P(S):=P_{\gamma,K}(S):=
\begin{vmatrix}
\mu_1 & \mu_2 & \mu_3 & \frac{r}{ K}(S- K)S \\
\alpha_1 & \alpha_2 & \alpha_3 & \frac{r}{ K}(S- K) \\
0 & \gamma_1 & \eta_1 & \mu_1-\alpha_1S \\
\gamma_2 & 0 & \eta_2 & \mu_2-\alpha_2S \\
\end{vmatrix}.
\end{equation*}
One can verify that
\begin{equation*}\label{p0p1}
P(S)=p_2S^2+p_1S+p_0,
\end{equation*}
where
\begin{eqnarray*}
&&p_0=r(-A_3\Delta_\mu-
\theta + \gamma_1\mu_2A_1+\gamma_2\mu_1A_2),
\\
&&p_1=r(A_3\Delta_\alpha+
\frac{\theta}{ K}+\rho-\gamma_1\alpha_2A_1-\gamma_2\alpha_1A_2),
\\
&&p_2=-\frac{r}{ K}\rho.
\end{eqnarray*}
Here
\begin{eqnarray*}\label{rho}
&&\rho:=
\begin{vmatrix}
\alpha_1 & \alpha_2 & \alpha_3  \\
0 & \gamma_1 & \eta_1  \\
\gamma_2 & 0 & \eta_2 \\
\end{vmatrix}=
\gamma_1
\alpha_1\eta_2+\gamma_2
\alpha_2 \eta_1 -\gamma_1
\gamma_2\alpha_3,\\
\label{theta}
&&\theta:=
\begin{vmatrix}
\mu_1 & \mu_2 & \mu_3  \\
0 & \gamma_1 & \eta_1  \\
\gamma_2 & 0 & \eta_2 \\
\end{vmatrix}=
\gamma_1
\mu_1\eta_2+\gamma_2
\mu_2 \eta_1 -\gamma_1
\gamma_2\mu_3,
\end{eqnarray*}
If the $S$ component is known the other components can easily be solved from the linear system of equations that results from the first three equations of \eqref{Equilib}.
\newline \newline
Let us introduce the Jacobian matrix of the right hand side of \eqref{submodel2}, with the redundant last row removed, evaluated at an inner equilibrium point $G_8=(S,I_1,I_2,I_{12})$:
\begin{equation*}\label{eq: J_8 expressed with B}
J_8={\rm diag}(S,I_1,I_2,I_{12})B,\;\;\;B=\left(\begin{matrix}
-\frac{r}{ K} & -\alpha_1 & -\alpha_2 & -\alpha_3  \\
\alpha_1 & 0 & -\gamma_1 & -\eta_1 \\
\alpha_2 & -\gamma_2 & 0 & -\eta_2 \\
\alpha_3 & \eta_1+\overline{\gamma}r_2 & \eta_2+\overline{\gamma}r_1 & -\overline{\gamma}r_1r_2
\end{matrix}
\right),
\end{equation*}
where
\begin{equation}\label{r1r2}
r_1=\frac{I_1}{I_{12}},\;\;\;r_2=\frac{I_2}{I_{12}}.
\end{equation}
Adding the first three rows of $J_8$ to its last row one obtains applying systematically \eqref{Equilib} that
\begin{align*}
\det B&=\frac{\det(J_8)}{SI_1I_2I_{12}}
=\frac{1}{I_{12}}\begin{vmatrix}
\frac{r}{ K} & \alpha_1 & \alpha_2 & \alpha_3 \\
-\alpha_1 & 0 & \gamma_1 & \eta_1 \\
-\alpha_2 & \gamma_2 & 0 & \eta_2 \\
\frac{r}{ K}(2S- K) & \mu_1 & \mu_2 & \mu_3
\end{vmatrix}
=\frac{1}{I_{12}}\begin{vmatrix}
 \mu_1 & \mu_2 & \mu_3 & \frac{r}{K}(2S- K)
\\
\alpha_1 & \alpha_2 & \alpha_3  & \frac{r}{ K}  \\
 0 & \gamma_1 & \eta_1 & -\alpha_1 \\
\gamma_2 & 0 & \eta_2 & -\alpha_2
\end{vmatrix}\\
&=\frac{1}{I_{12}}\,\frac{\partial P(S)}{\partial S}.
\end{align*}
The last equality is verified directly by using the definition of $P(S)$.
This implies an important property
\begin{equation}\label{May1a}
\det B=\frac{1}{I_{12}}\frac{\partial P(S)}{\partial S}.
\end{equation}

We assume that  
 \begin{equation}\label{May1aa}
\frac{\partial P_{\gamma,K}(S)}{\partial S}>0\;\;\mbox{for any coexistence equilibrium point}.
\end{equation}
\begin{remark}
Inequality \eqref{May1aa} together with (\ref{May1a}) implies, in particular, that the Jacobian matrix is invertible at every coexistence equilibrium and so there exist a curve $G(K)$ through this point, parameterized by $K$ and consisting of equilibrium points satisfying (\ref{May1aa}). Moreover (\ref{May1aa}) implies that the product of all eigenvalues of the Jacobian matrix at a coexistence eq. point is positive which is in agreement with the local stability of the corresponding equilibrium point. By Lemma \ref{lem:contG_8} and \eqref{etaA13} we have that $\Delta_\alpha>0$ if the condition (\ref{May1aa}) is valid and the set of coexistence equilibria is non empty. Then
since
 $$
 \frac{\partial P(S)}{\partial S}=\alpha_1\alpha_2(\sigma_2-\sigma_1)\Delta_\alpha+O(\bar{\gamma}),
 $$
inequality 	(\ref{May1aa}) is true for small $\bar{\gamma}$.
\end{remark}

\begin{lemma}\label{lemma: interval properties og G_8} Let $G(K)=(S(K),I_1(K),I_2(K),I_{12}(K))$ be a curve consisting of coexistence equilibrium points satisfying (\ref{May1aa}). Let also $(K_1,K_2)$ be the maximal interval of existence of such curve. Then

(i) $\frac{\partial S}{\partial K}<0$ and $\frac{\partial I_{12}}{\partial K}<0$ for $K\in (K_1,K_2)$.

(ii) $K_1\geq\sigma_1$ and there exists the limit $\lim_{K\to K_1}G(K)$ which is an equilibrium point with at least one zero component.

(iii) if $K_2<\infty$ then there exists the limit $\lim_{K\to K_2}G(K)$ which is an equilibrium point with at least one zero component.

(iv) if $K=\infty$ then there is a limit $\lim_{K\to \infty}G(K)$ which is an equilibrium point of the limit system ($K=\infty$).
\end{lemma}
\begin{proof}
Differentiating \eqref{Equilib1} with respect to $K$, we get
\begin{equation*}\label{Bmatrix}
    B\left(\begin{matrix}
    \dot{S}
    \\
    \dot{I}_1
    \\
    \dot{I}_2
    \\
    \dot{I}_{12}
    \end{matrix}\right)
    =
    \left(\begin{matrix}
   -\frac{rS}{K^2}
 \\
 0
 \\
 0
 \\
 0
    \end{matrix}\right).
 \end{equation*}
Therefore
$$
\dot{S}=-(B^{-1})_{11}\frac{rS}{K^2}=-\frac{rSI_{12}}{K^2\frac{\partial P(S)}{\partial S}}\cdot (\gamma_1\gamma_2\overline{\gamma}r_1r_2
+\gamma_1\eta_2(\eta_1+\overline{\gamma}r_2)+\eta_1\gamma_2(\eta_2+\overline{\gamma}r_1))
$$
and
$$
\dot{I}_{12}=-(B^{-1})_{41}\frac{rS}{ K^2}=-\frac{rSI_{12}}{ K^2\frac{\partial P(S)}{\partial S}}\cdot ({\alpha_1\gamma_2(\eta_2+\overline{\gamma}r_1)+\gamma_1\alpha_2(\eta_1+\overline{\gamma}r_2)+\gamma_1\gamma_2\alpha_3}).
$$
which proves (i).

To prove (ii) we note first that the equilibrium point $G_2$ is globally stable for $K\in (0,\sigma_1)$ according to \cite{part1}, Proposition 2, and therefore
$K_1\geq\sigma_1$. Next, since $S$ and $I_{12}$ components are monotone according to (i),  and bounded there is a limit
$$
S^{(1)}=\lim_{K\to K_1}S(K)\;\;\mbox{and}\;\;I_{12}^{(1)}=\lim_{K\to K_1}I_{12}(K).
$$
The $I_1$ and $I_2$ components satisfying equations
\begin{eqnarray*}
&&\gamma_1I_2=\alpha_1S - \eta_1I_{12} - \mu_1,\\
&&\gamma_2I_1=\alpha_2S - \eta_2I_{12}- \mu_2,
\end{eqnarray*}
which implies convergence of these components to $I_{1}^{(1)}$ and $I_{2}^{(1)}$ respectively as $K\to K_1$. Clearly
$G^{(1)}=(S^{(1)},I_{1}^{(1)},I_{2}^{(1)},I_{12}^{(1)})$ is an equilibrium point which must be on the boundary of the positive octant, otherwise one can continue the branch $G(K)$ outside the maximal interval of existence. This argument proves (ii). Proof of (iii) and (iv) are the same up to  some small changes as the proof of (ii).

\end{proof}

To exclude from our analysis the equilibrium point $G_4$ we will require in this text that
\begin{equation}\label{M3a}
\gamma^*<1.
\end{equation}
Under this condition $G_4$ is always unstable. Since we are interested only in locally stable equilibrium point the point $G_4$ will not appear in our forthcoming analysis.

\subsection{The equilibrium state $G_7$}\label{section the equilibrium state G7}
Let us consider the equilibrium point $G_7$. The components are given by proposition \ref{pro:equil} as
\begin{align*}
G_7&=(S^*,0,I_2^*,I_{12}^*),\\
S^*&=K(1-\frac{1}{\eta_2^*}),\\
I_2^*&=\frac{\alpha_3}{\eta_2}(\sigma_3-S^*),\\
I_{12}^*&=\frac{\alpha_2}{\eta_2}(S^*-\sigma_2).
\end{align*}
This point has type $G_7$ (i.e. three positive components) if and only if
\begin{equation*}\label{G7sstar}
\sigma_2<S^*<\sigma_3\quad\text{and}\quad \eta_2^*>1,
\end{equation*}
where the first relation  is equivalent to
\begin{equation}\label{Au11a}
\frac{\sigma_2\eta_2^*}{\eta_2^*-1}<K<\frac{\sigma_3\eta_2^*}{\eta_2^*-1}.
\end{equation}
Similarly to above we find the Jacobian matrix evaluated at $G_7$ as
$$
J_7=
\begin{bmatrix}
-r\frac{S^*}{K} & -\alpha_1S^* & -\alpha_2S ^*& -\alpha_3S^*  \\
0 & \alpha_1S^*-\eta_1I_{12}^*-\gamma_1I_2^*-\mu_1 & 0 & 0 \\
\alpha_2I_2^* & -\gamma_2I_2^* & 0 & -\eta_2I_2^* \\
\alpha_3I_{12}^* & \eta_1I_{12}^*+\overline{\gamma}I_2^*& \eta_2I_{12}^* & 0
\end{bmatrix}.
$$
where $S,I_2,I_{12}$ are given by Proposition~\ref{pro:equil}.
Since  the submatrix
\begin{equation*}\label{JII7}
\tilde J=
\begin{bmatrix}
-r\frac{S^*}{K} & -\alpha_2S^* & -\alpha_3S^*  \\
\alpha_2I_2^* & 0 & -\eta_2I_2^* \\
\alpha_3I_{12}^* & \eta_2I_{12}^* & 0
\end{bmatrix}
=
\begin{bmatrix}
S^* & 0 & 0 \\
0 & I_2^* & 0 \\
0 & 0 & I_{12}^*
\end{bmatrix}
\begin{bmatrix}
-\frac{r}{K} & -\alpha_2 & -\alpha_3  \\
\alpha_2 & 0 & -\eta_2 \\
\alpha_3 & \eta_2 & 0
\end{bmatrix},
\end{equation*}
is stable by Routh-Hurwitz criteria, we conclude that the matrix $J_7$ is stable whenever
\begin{equation}\label{stabcond7.1}
\alpha_1S^*-\eta_1I_{12}^*-\gamma_1I_2^*-\mu_1 <0.
\end{equation}
Using Proposition~\ref{pro:equil}, we can rewrite \eqref{stabcond7.1} as
\begin{equation}\label{equation for stability of G_7 before defining S_2}
 S^*(\Delta_\alpha-\gamma_1\alpha_3)>\Delta_\mu-\gamma_1\mu_3.
\end{equation}
If $\Delta_\alpha-\gamma_1\alpha_3=0$ then the linear stability holds whenever $\Delta_\mu-\gamma_1\mu_3<0$.
For $\Delta_\alpha-\gamma_1\alpha_3\neq 0$, let us define
\begin{equation*}\label{S2}
    \hat S_2=\frac{\Delta_\mu-\gamma_1\mu_3}{\Delta_\alpha-\gamma_1\alpha_3}\;\;\mbox{and}\;\;\hat{K}_2=\hat {S}_2\frac{\eta_2^*}{\eta_2^*-1}
\end{equation*}
then \eqref{equation for stability of G_7 before defining S_2} can be written
\begin{equation*}
\left\{\begin{split}
        S^*>\hat S_2\text{ if }\Delta_\alpha-\gamma_1\alpha_3>0
        \\
        S^*<\hat S_2\text{ if }\Delta_\alpha-\gamma_1\alpha_3<0
\end{split}\right.
\end{equation*}


It can be verified also that
\begin{align}
\hat S_2-\sigma_1 &=\frac{rA_1A_3(\eta_1^*-\gamma^*)}{(\Delta_\alpha-\gamma_1\alpha_3)\alpha_1}\nonumber\\
\label{S2sigma2}
    \hat S_2-\sigma_2 &=\frac{rA_2A_3(\eta_2^*-\gamma^*)}{(\Delta_\alpha-\gamma_1\alpha_3)\alpha_2}\\
   \hat S_2-\sigma_3 &=\frac{rA_1A_2(\eta_2^*-\eta_1^*)}{(\Delta_\alpha-\gamma_1\alpha_3)\alpha_3}\nonumber
   \end{align}
This readily yields the (local) stability criterion:

\begin{proposition}
\label{pro:G7}
The equilibrium point $G_7$ is nonnegative and locally stable if and only if $\eta_2^*>1$ and exactly one of the following conditions holds:
\begin{enumerate}[label=(\roman*),itemsep=0.4ex,leftmargin=0.8cm]

\item\label{G7-1}
 $\frac{\sigma_2\eta_2^*}{\eta_2^*-1}<K<\frac{\min(\hat S_2,\sigma_3)\eta_2^*}{\eta_2^*-1}$ when $\Delta_\alpha-\gamma_1\alpha_3<0$ and $\eta_2^*<\gamma^*$,
\item\label{G7-2}
 $\frac{\max(\hat S_2,\sigma_2)\eta_2^*}{\eta_2^*-1}<K<\frac{\sigma_3\eta_2}{\eta_2-1}$  when $\Delta_\alpha-\gamma_1\alpha_3>0$ and  $\eta_1^*>\eta_2^*$,
 \item\label{G7-3} $K$ subject to \rm{(\ref{Au11a})} when $\Delta_\alpha-\gamma_1\alpha_3=0$ and $\Delta_\mu-\gamma_1\mu_3<0$
    \end{enumerate}
    \end{proposition}

By \eqref{S2sigma2} we get that for small values of $\gamma_1$ one has the following refinement of the above proposition.

    \begin{corollary}\label{cor:G7stable}
    Let $\eta_2^*>1$ and
    \begin{equation}\label{gamma1le}
    0\le \gamma^*< \eta_2^*.
    \end{equation}
    Then the equilibrium point $G_7$ is nonnegative and linearly stable  if and only if

    (i) $\eta_1^*>\eta_2^*>1$ and $\frac{\hat{S}_2\eta_2^*}{\eta_2^*-1}<K<\frac{\sigma_3\eta_2^*}{\eta_2^*-1}$,

    or

    (ii) $K$ subject to \rm{(\ref{Au11a})}, $\Delta_\alpha-\gamma_1\alpha_3=0$ and $\Delta_\mu-\gamma_1\mu_3<0$.

    Therefore the bifurcation point $\hat{K}_2$ appears here only in the case (i).
    \end{corollary}

\begin{proof} By the made assumption, the case \ref{G7-1} in Proposition~\ref{pro:G7} is impossible. So it is sufficient to prove (i).
 It is thus required that $\eta^*_1>\eta_2^*>1$. If $\eta^*_1>\eta_2^*>1$ then
\begin{align*}
&\Delta_\alpha-\gamma_1\alpha_3=\eta_1^*A_1\alpha_2-\eta^*_2A_2\alpha_1-\gamma^*A_3\alpha_3
\\
&>\eta_1^*(
A_1\alpha_2-A_2\alpha_1-A_3\alpha_3)=0
\end{align*}
where we used equation $(\ref{etaA123})$ in the last equality.
Furthermore, since \eqref{gamma1le} and $\Delta_\alpha-\gamma_1\alpha_3>0$ holds for this case we also obtain from \eqref{S2sigma2} that $ \hat{S}_2-\sigma_2>0$, therefore $\max(\hat S_2,\sigma_2)=\hat{S}_2$, and we arrive at the desired conclusion.
\end{proof}

In what follows we will assume that
\begin{equation}\label{Ko1}
\gamma^*<1\;\;\mbox{and}\;\;\gamma_1<\alpha_3^{-1}\Delta_\alpha.
\end{equation}
We note that the first inequality guarantees (\ref{gamma1le}) since the equilibrium point $G_7$ exists only if $\eta_2^*>1$.

\section{Branches of coexistence equilibrium points}

\subsection{Bifurcation of $G_6$}\label{sec:g6}

From \cite{part1} we know that the equilibrium point  $G_6$ with the only zero component $I_2$ has the form
\begin{equation}\label{equation G_6 coordinates}
    G_6=(S^*,\frac{\alpha_3}{\eta_1}(\sigma_3-S^*),0,\frac{\alpha_1}{\eta_1}(S^*-\sigma_1))
\end{equation}
where
$$
S^*=K(1-\frac{1}{\eta^*_1}).
$$
We also know that it has positive components (except $I_2$) when $\eta^*_1>1$ and
\begin{equation*}\label{intervalforG_6}
    \sigma_1<S^*<\sigma_3\;\;\mbox{or equivalently}\;\;\frac{\sigma_1\eta^*_1}{\eta_1^*-1}< K<\frac{\sigma_3\eta^*_1}{\eta_1^*-1}
\end{equation*}
The bifurcation point (the point where the Jacobian is zero) corresponds to
\begin{equation}\label{Apr30b}
K=\hat K_1=\frac{\Delta_\mu+\mu_3\gamma_2}{\Delta_\alpha+\alpha_3\gamma_2}\frac{\eta_1^*}{\eta_1^*-1}\;\;\mbox{and}\;\;
S^{*}=\hat S_{1}=\frac{\Delta_\mu+\mu_3\gamma_2}{\Delta_\alpha+\alpha_3\mu_2}.
\end{equation}
and is denoted $    \hat G_6=G(\hat K_1)$.

Stability analysis of $G_6$ is given in the next proposition
\begin{proposition}
\label{pro:G6}
The equilibrium point $G_6$ is nonnegative if $\eta_1^*>1$ and it is stable if the following conditions hold:
\begin{equation*}\label{conG6}
\frac{\sigma_1\eta_1^*}{\eta_1^*-1}<K<\frac{Q\eta_1^*}{\eta_1^*-1}
\end{equation*}
where
\begin{equation*}\label{Qdef}
Q=\left\{
\begin{array}{ll}
\sigma_3& \text{\quad if  $\eta_2^*> \eta_1^*$};\\
\hat{S}_1& \text{\quad if $\eta_2^*<\eta_1^*$.}
\end{array}
\right.
\end{equation*}

\end{proposition}
The case $\eta_2^*> \eta_1^*$ (when we have no bifurcation) is considered in our paper \cite{part1}.
Here we will assume  that
\begin{equation}\label{Apr30a}
\eta^*_1>\eta^*_2.
\end{equation}
By \eqref{etaA13} the last inequality implies that

\begin{equation*}\label{Apr30aa}
\Delta_\alpha>(\eta^*_1-\eta_2^*)A_2\alpha_1>0.
\end{equation*}

By using \eqref{deltamu2}, one verifies straightforward that
$$
\sigma_2<\frac{\Delta_\mu}{\Delta_\alpha}<\hat S_{1}<\sigma_3.
$$

Notice a useful identity (the last equality is by \eqref{etaAetaA})
\begin{equation}\label{S1S2}
\hat{S}_1-\hat{S}_2=\frac{\alpha_3\bar \gamma (\sigma_3\Delta_\alpha-\Delta_\mu)}
{(\Delta_\alpha-\gamma_1\alpha_3)(\Delta_\alpha+\gamma_2\alpha_3)}=
\frac{\bar \gamma r A_1A_2(\eta_1^*- \eta_2^*)}
{(\Delta_\alpha-\gamma_1\alpha_3)(\Delta_\alpha+\gamma_2\alpha_3)}.
\end{equation}
As a result of Corollary \ref{cor:G7stable} and proposition \ref{pro:G6} we get that $\hat S_1,\hat S_2$ only exist as parts of equilibrium points when $\Delta_\alpha-\gamma_1\alpha_3>0$ and $\eta_1^*>\eta_2^*>1$. In this case we see from \eqref{S1S2} that $\hat S_1>\hat S_2$.

We will now prove the following lemma
\begin{lemma}\label{lemma about bifurcation at S_1} Let (\ref{Apr30a}) be valid and $\frac{\partial P}{\partial S}|_{S=\hat S_1}>0$.
Then  there exists a smooth branch of equilibrium points $G_8(K)=(S,I_1,I_2,I_{12})(K)$ defined for small $|K- \hat K_1|$ with the asymptotics
\begin{eqnarray}\label{equation I_2 of G_8 at bifurcation}
&&S(K)=\hat S_1+\mathcal{O}(K-\hat K_1)\nonumber\\
&&     I_1(K)=\frac{\alpha_3}{\eta_1}(\sigma_3-\hat S_1)+\mathcal{O}(K-\hat K_1)\nonumber\\
 && I_2(K)=\frac{\eta_1r(\Delta_\mu+\gamma_2\mu_3)I_{12}(\hat K_1)}{\frac{\partial P(S)}{\partial S}|_{K=\hat K_1} \hat K_1^2}(K-\hat K_1) +\mathcal{O}((K-\hat K_1)^2)\label{eq: I_2 of G_8 around K_1}\\
  &&  I_{12}(K)=\frac{\alpha_1}{\eta_1}(\hat S_1-\sigma)+\mathcal{O}(K-\hat K_1).\nonumber
\end{eqnarray}
 Furthermore this equilibrium point is locally stable for $\hat K_1<K\leq \hat K_1+\varepsilon$, where $\varepsilon$ is a small positive number. Moreover all equilibrium points 
 in a small neighborhood of $G_6(\hat{K}_1)$ are exhausted by two branches $G_6(K)$ and $G_8(K)$.
\end{lemma}

\begin{remark}
The constant $\varepsilon$ does not depend on $\gamma_i$ but it does depend on $\alpha_i$, $\mu_i$ and $\eta_i$.
\end{remark}

\begin{proof}
In order to use results of Appendix \ref{section appendix degenerate bifurcation point} we write system \eqref{Equilib} in the following form
\begin{eqnarray}\label{g6bifurcationcompactsystemform}
&&F(x;s)=0,\nonumber\\
&&x_4f(x')=0,
\end{eqnarray}
where
\begin{equation*}
x=(x',x_4)=(x_1,x_2,x_3,x_4)=(S,I_1,I_{12},I_2),
\end{equation*}
where $s=K-\hat K_1$,
\begin{equation*}
F(x,s)=\left(
\begin{matrix}
\Big(\frac{r}{K}(K-x_1)-\alpha_1x_2-\alpha_2x_4-\alpha_3x_3\Big)x_1
\\
(\alpha_1x_1-\eta_1x_3-\gamma_1x_4-\mu_1)x_2
\\
(\alpha_3x_1+\eta_1x_2+\eta_2x_4-\mu_3)x_3+\bar{\gamma}x_2x_4
\end{matrix}\right),
\end{equation*}
and
\begin{equation*}
f(x')=\alpha_2x_1-\eta_2x_3-\gamma_2x_2-\mu_2.
\end{equation*}
By the definition of the bifurcation point (\ref{Apr30b}) and (\ref{equation G_6 coordinates}), we have that
$$
x^*=(\hat S_1,\frac{\alpha_3}{\eta_1}(\sigma_3-\hat S_1),\frac{\alpha_1}{\eta_1}(\hat S_1-\sigma_1))
$$
solves the equation $F(x^*,0;0)=0$ and $f(x^*)=0$. Futhermore, the vector
$$
\xi(s)=x^*+s\Big(1-\frac{1}{\eta_1^*}\Big)\Big(1,\frac{-\alpha_3}{\eta_1},\frac{\alpha_1}{\eta_1}\Big)
$$
solves $F(\xi(s),0;s)=0$. The matrix $A=(\partial_{x_j}F_k(x^*,0,0))_{j,k=1}^{3}$ is evaluated as
$$
A={\rm diag}(x_1^*,x_2^*,x_3^*)\hat{A},\;\;\quad
\hat{A}=\left(
\begin{matrix}
-\frac{r}{\hat K_2} & -\alpha_1 & -\alpha_3
\\
\alpha_1 & 0 & -\eta_1
\\
\alpha_3 & \eta_1 & 0
\end{matrix}
\right).
$$
With the help of Hurwitz stability criterion we can deduce that $A(y^*,0)$ is stable and invertible.

 Since
 $$
 \nabla_{x'} f(x^*)=(\alpha_2,-\gamma_2,\eta_2),\;\;\partial_{x_4}F(x^*,0;0)=
 (-\alpha_2x^*_1,-\gamma_1x^*_2,\eta_2x^*_3+\bar{\gamma}x^*_2)^T,
 $$
 we get
$$
\Theta=\nabla_{x'} f \cdot A^{-1}\partial_{x_4}F|_{x=(x^*,0),s=0}=(\alpha_2,-\gamma_2,\eta_2)\hat{A}^{-1}
(-\alpha_2,-\gamma_1,\eta_2+\bar{\gamma}x^*_2/x_3^*)^T.
$$

Let us evaluate $\Theta$ and check that $\Theta\neq 0$.
First we observe the equality
\begin{equation}\label{M2a}
\left(
\begin{matrix}
1 & 0 & 0 & 0
\\
0 & 1 & 0 & 0
\\
0 & 0 & 0 & 1
\\
0 & 0 & 1 & 0
\end{matrix}
\right)\!\!\!
B(y^*,K_1)\!\!\!
\left(
\begin{matrix}
1 & 0 & 0 & 0
\\
0 & 1 & 0 & 0
\\
0 & 0 & 0 & 1
\\
0 & 0 & 1 & 0
\end{matrix}
\right)=\!\!
\left(\begin{matrix}
 &  &  & -\alpha_2
\\
 & \hat A &  & -\gamma_1
\\
 &  &  & \eta_2+\bar{\gamma}x^*_2/x_3^*
\\
\,\alpha_2 \,&\,-\gamma_2\,&\,\eta_2\, &\, 0
\end{matrix}
\right)
\end{equation}
By (\ref{May1a})
\begin{equation}\label{M2aaz}
\det (\mbox{left-hand side of (\ref{M2a})})=\frac{1}{I_{12}}\frac{\partial P(S)}{\partial S}\Big|_{S=S_1}.
\end{equation}
Let us show that
\begin{equation}\label{M2aa}
\det (\mbox{right-hand side of (\ref{M2a})})=\Theta\det\hat{A}\,.
\end{equation}
For this purpose consider the equation
\begin{equation}\label{ssim}
\left(\begin{matrix}
\hat A & (-\alpha_2,-\gamma_1,\eta_2+\bar{\gamma}x^*_2/x_3^*)^T
\\
(\alpha_2,-\gamma_2,\eta_2) & 0
\end{matrix}
\right)
\left(\begin{matrix}
X
\\
x
\end{matrix}\right)=
\left(
\begin{matrix}
\bar 0
 \\
1
\end{matrix}
\right)
\end{equation}
where $X\in\mathbb{R}^{3}$, $x\in \mathbb{R}$ and $\bar 0=(0,0,0)^T$.
We denote by $\hat{B}$ the matrix in the left-hand side of (\ref{ssim}) and using the expression for
 the matrix inverse, we get
\begin{equation}\label{M2b}
x=\frac{\det(\hat A)}{\det(\hat B)}.
\end{equation}
Solving (\ref{ssim}) as a linear system by finding first $X$ and then $x$ from the last equation, we obtain
$-\Theta x=1$. The last relation together with (\ref{M2b}) gives (\ref{M2aa}). Now the relations (\ref{M2aaz}) and
(\ref{M2aa}) imply
\begin{equation*}
\Theta=-\frac{\frac{\partial P(S)}{\partial(S)}\big|_{S=\hat S_1}}{\det(\hat A)I_{12}},
\end{equation*}
which along with

\begin{equation*}
\det(\hat A(\hat K_1))=\frac{1}{\hat S_1 I_1^* I_{12}^*}\left|
\begin{matrix}
-\frac{r}{\hat K}\hat S_1 & -\alpha_1\hat S_1 & -\alpha_3\hat S_1
\\
\alpha_1I_1^* & 0 & -\eta_1I_1^*
\\
\alpha_3I^*_{12} & \eta_1I^*_{12} & 0
\end{matrix}
\right|=-\frac{r}{\hat K_1}\eta_1^2
\end{equation*}
gives
$$
\Theta=\hat K_1\frac{\frac{\partial P(S)}{\partial(S)}\big|_{S=\hat S_1}}{r\eta_1^2I_{12}}>0.
$$
Next, since $f(x^*,0;0)=0$, we have
$$
f(\xi(s),0;s)=\frac{s}{\eta_1}\Big(1-\frac{1}{\eta_1^*}\Big)\Big(\Delta_\alpha+\gamma_2\alpha_3\Big).
$$
Now applying \eqref{Apr18d} in the appendix we get
\begin{equation*}
x_4=\frac{1}{\eta_1}\Big(1-\frac{1}{\eta_1^*}\Big)\Big(\Delta_\alpha+\gamma_2\alpha_3\Big)\frac{1}{\Theta}s+O(s^2),
\end{equation*}
which is equivalent to \eqref{eq: I_2 of G_8 around K_1}.

To prove local stability let us consider the matrix
$$
{\mathcal B}=\left(\begin{matrix}
 A & \partial_{x_n}F(x^*,0;0)
\\
0 & 0
\end{matrix}\right).
$$
Since the matrix $A$ is stable the matrix ${\mathcal B}$ has three eigenvalues with negative real part and one eigenvalue zero. The eigenvalues of the Jacobian matrix
$$
{\mathcal J}(s)=\left(\begin{matrix}
 \nabla_{x'}F(\hat{x}(s);s) & \partial_{x_n}F(\hat{x}(s);s)
\\
\nabla_{x'}(x_nf(x'))|_{x=\hat{x}(s)} & f(\hat{x}')
\end{matrix}\right)
$$
are small perturbation of the eigenvalue of ${\mathcal B}={\mathcal J}(s)$. Therefore three of them have negative real part for small $s$ and the last one $\lambda(\hat{x}(s))$, which is perturbation of zero eigenvalue of ${\mathcal B}$, has the following asymptotics (see \eqref{Apr20a} in the appendix)
$$
\lambda(\hat{x}(s))=-\frac{d}{ds}f(\xi(s),0;s)|_{s=0}s+O(s^2)
=-\frac{s}{\eta_1}\Big(1-\frac{1}{\eta_1^*}\Big)\Big(\Delta_\alpha+\gamma_2\alpha_3\Big)+O(s^2)
$$
and hence it is negative for small positive $s$. This proves the local stability of the coexistence equilibrium point.
\end{proof}

\subsection{Bifurcation of $G_7$}\label{section bifurcation of G7}
We will assume in this section that
\begin{equation}\label{eta1>eta2>1}
    \eta^*_1>\eta^*_2>1
\end{equation}
From \cite{part1} we know that the equilibrium point $G_7$ with only zero component $I_1$ has the form
\begin{equation}\label{equation G_7 coordinates}
    G_7=(S^*,0,\frac{\alpha_3}{\eta_2}(\sigma_3-S^*),\frac{\alpha_2}{\eta_2}(S^*-\sigma_2))
\end{equation}
where $S^*=K (1-\frac{1}{\eta^*_2})$.
We also know that it has positive components (except $I_1$) when $\eta^*_2>1$ and
\begin{equation*}
    \sigma_2<S^*<\sigma_3\text{ or equivalently  }
    \frac{\sigma_2\eta^*_2}{\eta_2^*-1}<K<\frac{\sigma_3\eta^*_2}{\eta_2^*-1} \label{intervalforG_7}
\end{equation*}

As in section \ref{sec:g6} we obtain
$$\Delta_\alpha>0$$

The bifurcation point (the point where the Jacobian is zero) corresponds to
\begin{equation}\label{eq: K_2,S_2}
    K=\hat K_2=\frac{\Delta_\mu-\gamma_1\mu_3}{\Delta_\alpha-\gamma_1\alpha_3}\frac{\eta^*_2}{\eta_2^*-1}\text{ and } S^*=\hat S_2=\frac{\Delta_\mu-\gamma_1\mu_3}{\Delta_\alpha-\gamma_1\alpha_3}
\end{equation}
and is denoted $    \hat G_7=G_7(\hat K_2)$.

For $\gamma_1^*<\eta_2^*$ we have that
$\hat S_2>\sigma_2$ and that $G_7$ is stable for $K$ in the interval
\begin{equation*}
    \frac{\hat S_2\eta^*_2}{\eta^*_2-1}<K<\frac{\sigma_3\eta_2^*}{\eta^*_2-1}
\end{equation*}

We will now prove the following lemma
\begin{lemma}\label{lemma linearization of G_8 at K_2}
Let \eqref{eta1>eta2>1} be valid and $\frac{\partial P}{\partial S}\big|_{S=\hat S_2}>0$. Then there exist a smooth branch of equilibrium points $G_8=(S,I_1,I_2,I_{12})(K)$ defined for small $|K-\hat K_2|$ with the asymptotics
\begin{align}
S(K)&=\hat S_2+\mathcal{O}(K-\hat K_2)\nonumber
\\
    I_1(K)&=-\frac{\eta_2r(\Delta_\mu-\gamma_1\mu_3)I_{12}}{\frac{\partial P(S)}{\partial S}\hat K_2^2}(K-\hat K_2) +\mathcal{O}((K-\hat K_2)^2)\label{equation I_1 of G_8 at bifurcation}
    \\
  I_2(K)&=\frac{\alpha_3}{\eta_2}(\sigma_3-\hat S_2)+\mathcal{O}(K-\hat K_2)\nonumber
    \\
    I_{12}(K)&=\frac{\alpha_1}{\eta_2}(\hat S_2-\sigma_2)+\mathcal{O}(K-\hat K_2).\nonumber
\end{align}
These equilibrium points are locally stable for $\hat K_2-\varepsilon\leq K< \hat K_2$, where $\varepsilon$ is a small positive number. Moreover all equilibrium points in a small neighborhood of $G_7(\hat{K}_2)$ 
 are exhausted by two branches $G_7(K)$ and $G_8(K)$.
\end{lemma}
\begin{remark}
The constant $\varepsilon$ does not depend on $\gamma$ but it does depend on $\alpha$, $\mu$ and $\eta$.
\end{remark}
\begin{proof}
We write system \eqref{Equilib} in the form
\begin{eqnarray*}
    &&F(x;s)
    \\
    &&f(x')=0,
    \end{eqnarray*}
    where
    $$x=(x',x_4)=(x_1,x_2,x_3,x_4)=(S,I_2,I_{12},I_1)$$
    where $s=K-\hat K_2$,
    \begin{equation*}
        F(x,s)=\left(\begin{matrix}
\Big(\frac{r}{K}(K-x_1)-\alpha_1x_4-\alpha_2x_2-\alpha_3x_3\Big)x_1
\\
(\alpha_2x_1-\eta_2x_3-\gamma_2x_4-\mu_2)x_2
\\
(\alpha_3x_1+\eta_1x_4+\eta_2x_2-\mu_3)x_3+\bar{\gamma}x_4x_2
\end{matrix}\right)
    \end{equation*}
    and
    $$
    f(x')=\alpha_1x_1-\eta_1x_3-\gamma_1x_2-\mu_1
    $$
    By the definition of the bifurcation point \eqref{eq: K_2,S_2} and \eqref{equation G_7 coordinates}, we have that
    $$x^*=(\hat S_2,\frac{\alpha_3}{\eta_2}(\sigma_3-\hat S_2), \frac{\alpha_{12}}{\eta_2}(\hat S_2-\sigma_2))$$
    solves the equation $F(x^*,0;0)=0$ and $f(x^*)=0$. Furthermore, the vector
    $$\xi(s)=x^*+s(1-\frac{1}{\eta_2^*})(1,\frac{-\alpha_3}{\eta_2},\frac{\alpha_2}{\eta_2})$$
    solves the equation $F(\xi(s),0,s)=0$. The matrix  $A=(\partial_{x_j}F_k(x^*,0,0))^3_{j,k=1}$ is evaluated as
    $$A=\text{diag}(x_1^*,x_2^*,x_3^*)\hat A,\; \hat A=\left(\begin{matrix}
    -\frac{r}{K} & -\alpha_" & -\alpha_3
    \\
    \alpha_2 & 0 & -\eta_2
    \\
    \alpha_3 & \eta_2 & 0
    \end{matrix}\right)$$
    with the help of Hurwitz stability criterion we can deduce that $A$ is stable and invertible.
    Since
    $$\nabla_{x'}f(x^*)=(\alpha_1,\gamma_1,\eta_1),\;\partial_{x_4}F(x^*,0,;0)=(-\alpha_1x_1^*,-\gamma_2x^*_2, \eta_1x_3+\bar \gamma x_2)^T$$
    we get
 $$\Theta=\nabla_{x'}f\cdot A^{-1}\partial_{x_4}F\big|_{x=(x^*,0)}=(\alpha_1,\gamma_1,\eta_1)\hat A^{-1}(-\alpha_1x_1^*,-\gamma_2x^*_2, \eta_1x_3+\bar \gamma x_2)^T$$
 Let us evaluate $\Theta$ and check that $\Theta\neq 0$. First we observe the equality
\begin{equation}\label{M2a}
\left(
\begin{matrix}
1 & 0 & 0 & 0
\\
0 & 1 & 0 & 0
\\
0 & 0 & 0 & 1
\\
0 & 0 & 1 & 0
\end{matrix}
\right)\!\!\!
B(y^*,\hat K_2)\!\!\!
\left(
\begin{matrix}
1 & 0 & 0 & 0
\\
0 & 1 & 0 & 0
\\
0 & 0 & 0 & 1
\\
0 & 0 & 1 & 0
\end{matrix}
\right)=
\left(\begin{matrix}
 &  &  & -\alpha_1
\\
 & \hat A &  & -\gamma_2
\\
 &  &  & \eta_1+\bar{\gamma}x^*_2/x_3^*
\\
\,\alpha_1 \,&\,-\gamma_1\,&\,\eta_1\, &\, 0
\end{matrix}
\right)
\end{equation}

In the same way as in section \ref{sec:g6} we get
\begin{equation*}
    \Theta=\hat K_2 \frac{\frac{\partial P(S)}{\partial S}\big|_{S=\hat S_2}}{\det(\hat A)I_{12}}
\end{equation*}
which along with
\begin{equation*}
    \det(\hat A(\hat K_2))=\left|
    \begin{matrix}
    -\frac{r}{\hat K_2} & -\alpha_2 & -\alpha_3
    \\
    \alpha_2 & 0 & -\eta_2
    \\
    \alpha_3 & \eta_2 & 0
    \end{matrix}
    \right|=-\eta_2^2\frac{r}{\hat K_2}
\end{equation*}
Next, since $f(x^*,0;0)=0$, we have
$$
f(\xi(s),0;s)=\frac{s}{\eta_2}\Big(1-\frac{1}{\eta_2^*}\Big)\Big(\Delta_\alpha-\gamma_1\alpha_3\Big).
$$
Now applying \eqref{Apr18d} we get
\begin{equation*}
x_4=\frac{1}{\eta_2}\Big(1-\frac{1}{\eta_2^*}\Big)\Big(\Delta_\alpha-\gamma_1\alpha_3\Big)\frac{1}{\Theta}s+O(s^2),
\end{equation*}
which is equivalent to \eqref{equation I_1 of G_8 at bifurcation}.
To prove local stability let us consider the matrix
$$
{\mathcal B}=\left(\begin{matrix}
 A & \partial_{x_n}F(x^*,0;0)
\\
0 & 0
\end{matrix}\right).
$$
Since the matrix $A$ is stable the matrix ${\mathcal B}$ has three eigenvalues with negative real part and one eigenvalue zero. The eigenvalues of the Jacobian matrix
$$
{\mathcal J}(s)=\left(\begin{matrix}
 \nabla_{x'}F(\hat{x}(s);s) & \partial_{x_n}F(\hat{x}(s);s)
\\
\nabla_{x'}(x_nf(x'))|_{x=\hat{x}(s)} & f(\hat{x}')
\end{matrix}\right)
$$
are small perturbation of the eigenvalue of ${\mathcal B}={\mathcal J}(s)$. Therefore three of them have negative real part for small $s$ and the last one $\lambda(\hat{x}(s))$, which is a perturbation of the zero eigenvalue of ${\mathcal B}$, has the following asymptotics (see \eqref{Apr20a} in \textcolor{teal}{appendix \ref{section appendix degenerate bifurcation point}})
$$
\lambda(\hat{x}(s))=-\frac{d}{ds}f(\xi(s),0;s)|_{s=0}s+O(s^2)
=-\frac{s}{\eta_2}\Big(1-\frac{1}{\eta_2^*}\Big)\Big(\Delta_\alpha-\gamma_1\alpha_3\Big)+O(s^2)
$$
and hence it is negative for small positive $s$. This proves the local stability of the coexistence equilibrium point.
\end{proof}


\subsection{Equilibrium transition for coexistence equilibrium points}

\begin{lemma}\label{lem:contG_8} Let the assumption (\ref{May1aa}) be valid. If there exist a coexistence equilibrium point then

(i)
$$
\eta_1^*>\eta_2^*\;\;\mbox{and}\;\; \eta_1^*>1
$$
and this point lies on the branch of coexistence eq. points which starts at $K=\hat K_1$ at the bifurcation point $\hat{G}_6$.
Moreover

(ii) if additionally $\eta_1^*>\eta_2^*>1$ then the above branch is finished at $K=\hat K_2$ at the point $\hat{G}_7$.

(iii) If $\eta_1^*>1>\eta_2^*$ then the above branch can be continued up to $K=\infty$.
\end{lemma}

\begin{proof} Let us assume that there is a coexistence eq. point $G^*_8$ for $K=K^*$. Let $(K_1,K_2)$ be the maximal existence interval for existence of the branch $G_8(K)$ of coexistence equilibrium points containing $K^*$ and $G_8(K^*)=G^*_8$. According to Lemma \ref{lemma: interval properties og G_8} there exists the limit $G^*=\lim_{K\to K_1}G_8(K)$ and this limit is an equilibrium with at least one zero component.
According Lemma 2 in \cite{part1} the only possible scenarios are either that $G^*=G_6$ and $\alpha_2S^* - \eta_2I_{12} - \gamma_2 I_1 - \mu_2 =0$ or that $G^*=G_7$ and $\alpha_1S-\eta_1I_{12}-\gamma_1 I_2-\mu_1=0$. This happens only if $G^*=\hat G_6$ or $G^*=\hat G_7$. The case $G^*=G_4$ is disregarded due to assumption \eqref{M3a}.
According to \eqref{S1S2} with its associated comment we have that $\hat S_1>\hat S_2$ and $\eta^*_1>\eta_2^*$. Since $\hat S_1>S_2$ and $\frac{\partial S}{\partial K}<0$ according to Lemma \ref{lemma: interval properties og G_8} deduce that $G^*=\hat G_6$.
From existence of $\hat G_6$ it follows that $\eta_1^*>1$ and we obtain (i).

If $\hat K_2$ is finite then there is a limit of $G_8(K)$ as $K\to \hat K_2$ and this limit lies on the boundary. Simple modification of the above arguments shows that this limit is $\hat{G}_7$ which gives (ii).

In the case (iii) there are no $\hat{G}_6$ or $\hat{G}_7$ and hence the branch can be continued for all $K>\hat{K}_1$.

\end{proof}



When considering coexistence equilibrium points we assume that:

{\bf Assumption II}
\begin{itemize}
\item[$\bullet$]  If  $\eta_1^*>\eta_2^*$ and $\eta_1^*>1$  then $\partial_SP(\hat{S}_1)>0$ when $K=\hat{K}_1$;

\item[$\bullet$]  If $\eta_1^*>\eta_2^*>1$ then additionally to (i) it is supposed that $\partial_SP(\hat{S}_2)>0$ when $K=\hat{K}_2$.
\end{itemize}

\begin{lemma}\label{lemma:exG8} (i) Let $\eta_1^*>\eta_2^*>1$ and $\partial_SP(\hat{S}_i)>0$ when $K=\hat{K}_i$, $i=1,2$. Then there is a branch of coexistence equilibrium points starting at $\hat{G}_6$, $K=\hat{K}_1$, and ending at $\hat{G}_7$, $K=\hat{K}_2$.
All possible coexistence equilibrium points lies on this branch.

(ii) Let $\eta_1^*>1>\eta_2^*$ and $\partial_SP(\hat{S}_1)>0$ when $K=\hat{K}_1$. There is a branch  of coexistence equilibrium points starting at $\hat{G}_6$, $K=\hat{K}_1$, and defined for all $K>\hat{K}_1$.
All possible coexistence equilibrium points lies on this branch.
\end{lemma}
\begin{proof} (i) By Lemma \ref{lemma linearization of G_8 at K_2} there is a branch of coexistence equilibrium points ending at $\hat{G}_7$, $K=\hat{K}_2$ and defined for small $\hat{K}_2-K>0$. By Lemma \ref{lem:contG_8} it can be continued to the interval $(\hat{K}_1,\hat{K}_2)$ and the limit when $K\to \hat{K}_1$ is equal to $\hat{G}_6$. If we take any coexistence equilibrium then by Lemma \ref{lem:contG_8} it must lie on an equilibrium curve starting at $\hat{G}_6$. Then by uniqueness in Lemma \ref{lemma about bifurcation at S_1} this curve must coincide with the coexistence equilibrium branch constructed in the beginning.

(ii) In this case there is no bifurcation point $\hat{G}_7$ and the proof repeats with some simplifications the proof of (i).

\end{proof}

\section{Stability of coexistence equilibrium points}

\subsection{Auxiliary assertion}

Let $Q$ and $q$ be two positive constants. We introduce the set of $Y=(Y_1,Y_2,Y_3,Y_4)$
\begin{equation*}\label{M14a}
{\mathcal Y}=\{Y:Y_k\leq Q \,  \min( Y_1,Y_4)\geq q,Y_2+Y_3\geq q,\;\min(Y_2, Y_3)\geq 0\}.
\end{equation*}
Consider the matrix depending on the parameters $Y$ and $K$:
\begin{equation*}\label{M14aa}
{\mathcal M}={\mathcal M}(Y,K)={\rm diag}(Y_1,Y_2,Y_3,Y_4)M,\;\;M=\left(\begin{matrix}-\frac{r}{K}&-\alpha_1&-\alpha_2&-\alpha_3\\
\alpha_1&0&0&-\eta_1\\
\alpha_2&0&0&-\eta_2\\
\alpha_3&\eta_1&\eta_2&0
\end{matrix}\right).
\end{equation*}
Let also $\lambda_k=\lambda_k(Y,K)$, $k=1,2,3,4$, be their eigenvalues numerated according to the order $|\lambda_1|\geq |\lambda_2|\geq |\lambda_3|\geq |\lambda_4|$.
In the next lemma we give some more information about the first three eigenvalues.
\begin{lemma}\label{lemma: biggest eigenvalue less then zero} Let $0<K_1<K_2$. Then

\begin{equation}\label{M15a1}
\Xi=\max_{j=1,2,3}\max_{Y\in {\mathcal Y}}\max_{K_1\leq K\leq K_2} \Re\lambda_j(Y,K)<0.
\end{equation}

\end{lemma}
\begin{proof}
First assume that all components of $Y$ are non-zero. Let $\lambda\in\Bbb C$ be an eigenvalue of ${\mathcal M}$, i.e.
\begin{equation}\label{M15a}
{\mathcal M}X=\lambda X,\;\;X=(X_1,X_2,X_3,X_4)^T\in\Bbb C^4,\;\;X\neq 0.
\end{equation}
This implies
\begin{equation*}
\Re ({\mathcal M}X,D^{-1} X)=-\frac{r}{K}|X_1|^2=\Re\lambda(D^{-1}X,X),\end{equation*}
where $D={\rm diag}(Y_1,Y_2,Y_3,Y_4)$ and $(\cdot,\cdot)$ is the inner product in $\Bbb C^4$. Therefore
$$
\Re\lambda=-\frac{r}{K}\frac{|X_1|^2}{(D^{-1}X,X)}.
$$
This gives $\Re\lambda\leq 0$. Assume now that $\lambda=i\tau$, $\tau\in\Bbb R$, which implies $X_1=0$. Then (\ref{M15a}) implies
\begin{eqnarray}\label{M15aa}
&&\alpha_1X_2+\alpha_2X_3+\alpha_3X_4=0\nonumber\\
&&-\eta_1Y_2X_4=\lambda X_2,\;\;-\eta_2Y_3X_4=\lambda X_3\nonumber\\
&&Y_4(\eta_1X_2+\eta_2X_3)=\lambda X_4.
\end{eqnarray}
If $\lambda=0$ then $X_4=0$ and from the first and last equations in (\ref{M15aa}) we get that $X_2=X_3=0$. If $X_4=0$ and $\lambda\neq 0$ then from the middle equations in (\ref{M15aa}) we obtain $X_2=X_3=0$. Consider the case when $\lambda\neq 0$ and $X_4\neq 0$. Expressing $X_2$ and $X_3$ through $X_4$ from the middle equations in (\ref{M15aa}) and putting them in the first equation, we get
$$
X_4\Big(-\frac{\alpha_1\eta_1Y_2+\alpha_2\eta_2Y_3}{\lambda}+\alpha_3\Big)=0,
$$
which implies $X_4=0$. Thus we have shown that there are no eigenvalues of ${\mathcal M}$ on the imaginary line, i.e. $\Re\lambda_j<0$, $j=1,2,3,4$, provided all $Y_j$ ar positive.

Next consider the case $Y_2=0$. Then one eigenvalue of ${\mathcal M}$ is zero and the remaining  three  can be found from the eigenvalue problem
\begin{equation}\label{M15b}
{\bf diag}(Y_1,Y_3,Y_4)\left(\begin{matrix}-\frac{r}{K}&-\alpha_2&-\alpha_3\\
\alpha_2&0&-\eta_2\\
\alpha_3&\eta_2&0
\end{matrix}\right)(X_1,X_3,X_4)^T=\lambda (X_1,X_3,X_4)^T.
\end{equation}
Similar to the eigenvalue problem (\ref{M15a}) one can show that $\Re\lambda <0$ for (\ref{M15b}).

The argument in the case $Y_3=0$ is the same as in the case $Y_2=0$. Thus we have shown that for all
$(Y,K)\in {\mathcal Y}$, $\Re\lambda_j(Y,K)<0$, $j=1,2,3$. Since the eigenvalues continuously depend on $(Y,K)$ and the set ${\mathcal Y}$ is compact, we arrive at (\ref{M15a1}).

\end{proof}

\subsection{Local stability in the case $\eta_1^*>\eta_2^*>1$}\label{sec:locstab}

The main stability result for the equilibrium points branch in Lemma \ref{lemma: biggest eigenvalue less then zero} is the following

\begin{proposition}\label{pro:eta_2-1} Let $\eta_1^*>\eta_2^*>1$ and $G_8(K)$, $\hat{K}_1\leq K\leq\hat{K}_2$ be the branch constructed in Lemma \ref{lemma:exG8} (i). Then there exists a constant $\delta$ depending only on $\alpha_j$, $j=1,2,3$, and $\eta_1$, $\eta_2$ such that if $\bar{\gamma}\leq \delta$ then all points on this branch for $\hat{K}_1< K<\hat{K}_2$ are locally stable.

\end{proposition}
\begin{proof} Consider equilibrium points $G_8(K)=(S(K),I_1(K),I_2(K),I_{12}(K))$, $K\in [\hat{K}_1,\hat{K}_2]$. By Corollary 2 of \cite{part1} and Lemma~\ref{lemma:exG8} all of the components are non-negative and bounded by a certain constant independent of $K$ and $\gamma_1,\,\gamma_2$. Solving for $S$ and $I_{12}$ in the second and third of \eqref{Equilib1} gives
\begin{equation}\label{M15ba}
S(K)=\frac{\Delta_\mu}{\Delta_\alpha}+O(\bar{\gamma}),\;\;I_{12}(K)=\frac{\alpha_1\alpha_2(\sigma_2-\sigma_1)}{\Delta_\alpha}
+O(\bar{\gamma}).
\end{equation}
Furthermore, from the last equation in \eqref{Equilib1} we get
$$
\eta_1I_1+\eta_2I_2=\alpha_3(\sigma_3-S)-\bar{\gamma}\frac{I_1I_2}{I_{12}}
$$
which implies due to \eqref{M15ba}
\begin{equation*}\label{M15bb}
\eta_1I_1+\eta_2I_2=\frac{rA_1A_2(\eta_1^*-\eta_2^*)}{\alpha_3\Delta_\alpha}+O(\bar{\gamma}).
\end{equation*}
We will keep the notation ${\mathcal Y}$ for our case ($Y=(S,I_1,I_2,I_{12})$), where
$$
q=\frac{1}{2}\min\Big(\frac{\Delta_\mu}{\Delta_\alpha},\frac{\alpha_1\alpha_2(\sigma_2-\sigma_1)}{\Delta_\alpha},
\frac{1}{\min(\eta_1,\eta_2)}\frac{rA_1A_2(\eta_1^*-\eta_2^*)}{\alpha_3\Delta_\alpha}\Big),
$$
$$
Q=\max\Big(\sigma_3,\frac{r}{\sigma_1}\Big),
$$
and $\bar{\gamma}$ is kept sufficiently small. Then we can use Lemma \ref{lemma: biggest eigenvalue less then zero}, where $K_j=\hat{K}_j$, $j=1,2$.

We introduce two subsets of ${\mathcal Y}\times [\hat{K}_1,\hat{K}_2]$. The first one $\hat{\mathcal Y}_1$ consists of all $(Y;K)\in {\mathcal Y}\times [\hat{K}_1,\hat{K}_2] $ such that $\Re\lambda_1\geq \Xi/2$, where $\Xi$ is the constant from Lemma \ref{lemma: biggest eigenvalue less then zero} (in our case it depends only on $\alpha_j$, $j=1,2,3,$ and $\eta_1$, $\eta_2$. The second set $\hat{\mathcal Y}_2$ consists of all $(Y;K)\in {\mathcal Y}\times [\hat{K}_1,\hat{K}_2] $ such that $\Re\lambda_1\leq \Xi/2$. Introduce the contours
$$
\Gamma_1=\{\lambda\in\Bbb C:\Re\lambda=\Xi/4, |\Im\lambda|\leq C,\;\lambda-\Xi/4=Ce^{i\varphi},\varphi\in (\pi/2,3\pi/2)\}
$$
and
$$
\Gamma_2=\{\lambda\in\Bbb C:\Re\lambda=3\Xi/4, |\Im\lambda|\leq C,\;\lambda-3\Xi/4=Ce^{i\varphi},\varphi\in (\pi/2,3\pi/2)\},
$$
where $C$ is sufficiently large.
Put
$$
a_k=\max_{\hat{\mathcal Y}_k}\max_{\lambda\in\Gamma_k}||({\mathcal M}-\lambda)^{-1}||,\;\;k=1,2.
$$

By Lemma 6 there are at least 3 eigenvalues of $\mathcal{M}$ with $\Re \lambda\leq \Xi$ Consider two cases $(i)$ the remaining eigenvalue satisfies $\Re \lambda<\frac{5}{8}\Xi$ or $(ii)$ it satisfies $\Re \lambda\geq \frac{5}{8}\Xi$. Since the norm of the matrix
\begin{equation*}
    \mathcal{N}=\mathrm{diag}(S,I_1,I_2,I_{12})\left(\begin{matrix}
    0 & 0 & 0 & 0
    \\
    0 & 0 & -\gamma_1 & 0
    \\
    0 & -\gamma_2 & 0 & 0
    \\
    0 & \bar \gamma r_2 & \bar \gamma r_1 & -\bar \gamma r_1 r_2
    \end{matrix}\right)
\end{equation*}
is estimated by $C_1\bar \gamma$ with $C_1$ independent on $\gamma$ and $K$ we conclude that by Rouche's theorem the number of eigenvalues inside $\Gamma_2$ of the matrix $\mathcal{M}$ and $\mathcal{M+N}$ is the same for small $\bar \gamma$ in the case (i). Similarly we have that in the case $(ii)$ the number of eigenvalues of $\mathcal{M}$ and $\mathcal{M+N}$ is the same inside the contour $\Gamma_1$ for small $\bar \gamma$ and this number is equal to $4$.

This implies that for small $\bar{\gamma}$ there are at least three eigenvalues of the Jacobian matrix $J_8$ with negative real part on the branch $G_8(K)$. Since $$\det J_8(G_8(K))>0 \quad \text{for $K\in (\hat{K}_1,\hat{K}_2)$}$$
we conclude that all eigenvalues of $J_8(G_8(K))$ must have negative real part. This proves the proposition.

\end{proof}

\subsection{Instability for large $K$}

In this section we assume that
\begin{equation}\label{M18a}
\eta_1^*>1>\eta_2^*.
\end{equation}
According to Lemma \ref{lemma:exG8} there exists a branch $G_8(K)$, $K\in [\hat{K}_1,\infty )$, of coexistence equilibrium points starting from $G_8( \hat{K}_1)=\hat{G}_6$.
For $K=\infty$ and $\gamma=0$ the interior point has the coordinates
\begin{equation*}
 G_8 (\infty)|_{\gamma=0}=  (S^*,I_1^*,I^*_2,I^*_{12})=\Big(\frac{\Delta_\mu}{\Delta_\alpha},\frac{r}{\Delta_\alpha}(A_2-\eta_2),
  \frac{r}{\Delta_\alpha}(\eta_1-A_1),\frac{r}{\Delta_\alpha}A_3\Big).
\end{equation*}
All eigenvalues of the corresponding Jacobian matrix lie on the imaginary axis.

For $K=\infty$ and small $\gamma>0$ the interior point has the coordinates $G_\infty(\gamma)= G_\infty (0)+O(|\gamma|)$, where $\gamma=(\gamma_1,\gamma_2)$. Our goal is to analyze the location of eigenvalues of the Jacobian matrix when $\gamma$ is small.

The characteristic polynomial of the Jacobian matrix of the interior point is (up to a positive factor)
$$
   \frac{1}{SI_1I_2I_{12}} \det(J_8-\lambda I)=p(\lambda):=\left|\begin{matrix}
    -\frac{\lambda}{S} & -\alpha_1 & -\alpha_2 & -\alpha_3
    \\
    \alpha_1 & -\frac{\lambda}{I_1} & -\gamma_1 & -\eta_1
    \\
    \alpha_2 & -\gamma_2 & -\frac{\lambda}{I_2} & -\eta_2
    \\
    \alpha_3 & \eta_1+\bar \gamma r_2 & \eta_2+\bar \gamma r_1 & -\bar \gamma r_1r_2-\frac{\lambda}{I_{12}}
    \end{matrix}\right|
$$
where $r_i$ are defined in \eqref{r1r2}. It is clear that the polynomial $p$ is monic. The necessary condition for stability of the polynomial $p$ is the positivity of all its coefficients. Let us evaluate  the coefficient $p_1$ of the $\lambda$ term and show that it can be negative for certain choice of  parameters  (we note that for $\gamma=0$ this coefficient is zero). This will imply that some of all eigenvalues must have positive real part. We have
\begin{align*}
  p_1=&-
    \Big(
  \frac{1}{S}
  \left|\begin{matrix}
  0 & -\gamma_1 & -\eta_1
  \\
  -\gamma_2 & 0 & -\eta_2
  \\
  \eta_1+\bar \gamma r_2 & \eta_2+\bar \gamma r_1 & -\bar \gamma r_1r_2
  \end{matrix}\right|+\frac{1}{I_1}
  \left|
  \begin{matrix}
  0 & -\alpha_2 & -\alpha_3
  \\
  \alpha_2 & 0 & -\eta_2
  \\
  \alpha_3 & \eta_2+\bar \gamma r_1 & -\bar \gamma r_1r_2
  \end{matrix}
\right|
\\
&+\frac{1}{I_2}\left|
\begin{matrix}
0 & -\alpha_1 & -\alpha_3
\\
\alpha_1 & 0 &-\eta_1
\\
\alpha_3 & \eta_1+\bar \gamma r_2 & -\bar \gamma r_1r_2
\end{matrix}
\right|
+\frac{1}{I_{12}}
\left|
\begin{matrix}
0 & -\alpha_1 & -\alpha_2
\\
\alpha_1 & 0 & -\gamma_1
\\
\alpha_2 & -\gamma_2 & 0
\end{matrix}
\right|
    \Big)
   \\
    &=\bar \gamma\left(
    -\frac{ \eta_1\eta_2}{S} + \frac{ r_1\alpha_2(\alpha_3+\alpha_2r_2)}{I_1} +\frac{ \alpha_1r_2(\alpha_1r_1+\alpha_3)}{I_2}-\frac{ \alpha_1\alpha_2}{I_{12}}
    \right)+O(\bar \gamma^2)
    \\
    &=\bar \gamma\left(
    -\frac{ \eta_1\eta_2}{S} + \frac{ \alpha_2(\alpha_3+\alpha_2r_2)+\alpha_1(\alpha_1r_1+\alpha_3)-\alpha_1\alpha_2}{I_{12}}
    \right)+O(\bar \gamma^2)
\end{align*}

Plugging in the values of $S,I_{12},r_1,r_2,$  for $K=\infty$ and $\gamma=0$ we continue the above equalities
\begin{align*}
  p_1=&
\bar \gamma\Delta_\alpha\left(-\frac{\eta_1\eta_2 }{\Delta_\mu}+\frac{(\alpha_2+\alpha_1)\alpha_3-\alpha_1\alpha_2}{rA_3}+\frac{\alpha_2^2(\eta_1-A_1)+\alpha_1^2(A_2-\eta_2)}{rA_3^2}\right)+O(\gamma^2)
\\
=&
\bar \gamma\Delta_\alpha\left(-\frac{\eta_1\eta_2}{\Delta_\mu}+\frac{\alpha_2^2\eta_1-\alpha_1^2\eta_2}{rA_3^2}+\frac{(\alpha_2+\alpha_1)\alpha_3-\alpha_1\alpha_2}{rA_3}+\frac{\alpha_1^2A_2-\alpha_2^2A_1}{rA_3^2}\right)+O(\gamma^2)
\end{align*}
with
\begin{align*}
    b-\mu_0=1
    \\
    \mu_i=1
    \\
    \alpha_1=10,\;\alpha_2=9.9,\;\alpha_3=1,
\end{align*}
 \eqref{sigma_i ordering} and \eqref{deltamu} is satisfied and we get
\begin{align*}
    A_1=9
    \\
    A_2=0.1
    \\
    A_3=8.9.
\end{align*}
We can now choose $\eta_1$ and $\eta_2$ such that $\eta_1*>1>\eta_2^*$ and $\eta_1^*,\eta_2^*\approx 1$. For these values Lemma~\ref{lem:contG_8} tells us that there exist a coexistence equilibrium branch defined for infinitely large $K$. On the other hand the coefficient of the $\lambda^1$ term is approximately
$$
    SI_1I_2I_{12}\gamma\Delta_\alpha(-\frac{10*9.9}{0.1}+\frac{19.9-9.9}{0.1}+0)=SI_1I_2I_{12}\gamma\Delta_\alpha(-890)<0
$$
and so when $K$ is sufficiently large and $\gamma$ is sufficiently small this equilibrium branch is unstable.

\subsection{Hopf bifurcation}

In this section we assume that (\ref{M18a}) is satisfied. Thus there exists a branch of coexistence equilibrium points $G_8(K)$ defined for $K>\hat{K}_1$. We assume also that the parameters $\alpha_1,\,\alpha_2,\,\alpha_3$ and $\eta_1,\,\eta_2$ are chosen such that the stability is lost when $K\bar \gamma$ is large.
Since the point $G_8(K)$ is stable when $K$ is close to $\hat K_1$ there exist a point $K=K_c$ where the local stability of $G_8$ is lost. Since the trace of the Jacobian matrix is always negative the eigenvalues can only reach the imaginary axis in pairs.
If we assume that the derivative of their real part at $K=K_c$ is positive then there is a simple Hopf bifurcation so for $K$ close to $K_c$ there are periodic oscillations, see \cite{liu1994criterion}.


\subsection{Local stability  in the case $\eta_1^*>1>\eta_2^*$}\label{sec: Local stability in the case eta_1^*>1>eta_2^*}
\begin{theorem}\label{the:Kgamma1} Let $\eta_1^*>1>\eta_2^*$ and let $G_8(K)$, $K\in (\hat{K}_1,\infty)$ be the branch of equilibrium points starting at $\hat{G}_6$.
There exists a constant $\omega>0$ depending on $\alpha_1,\alpha_2,\alpha_3$ and $\eta_1,\eta_2$ such that if $K\bar \gamma\leq\omega$ for $K\in (\hat{K}_1,\infty)$, then the inner equilibrium points $G_8(K)$ are locally stable.
\end{theorem}

\begin{proof} In what follows in the proof we will denote by $c$ and $C$, possibly with indexes, various positive constants depending on $\alpha_1,\,\alpha_2,\,\alpha_3$ and $\eta_1,\,\eta_2$. The Jacobian matrix is equal to $J_8(K)=D(A+K+\Gamma)$, where $D={\rm diag}(S,I_1,I_2,I_{12})$,
\begin{equation*}
    K\!=\!\left(
    \begin{matrix}
    -\frac{r}{K} & 0 & 0 & 0
    \\
    0 & 0 & 0 & 0
    \\
    0 & 0 & 0 & 0
    \\
    0 & 0 & 0 & -\bar \gamma r_1r_2
    \end{matrix}\right),\;\;
    A\!=\!\left(
    \begin{matrix}
    0 & -\alpha_1 & -\alpha_2 & -\alpha_3
    \\
    \alpha_1 & 0 & 0 & -\eta_1
    \\
    \alpha_2 & 0 & 0 & -\eta_2
    \\
    \alpha_3 & \eta_1 & \eta_2 & 0
    \end{matrix}
    \right),\;
\Gamma\!=\!\left(
\begin{matrix}
0 & 0 & 0 & 0
\\
0 & 0 & -\gamma_1 & 0
\\
0 & -\gamma_2 & 0 & 0
\\
0 & \bar \gamma r_2 & \bar \gamma r_1 & 0
\end{matrix}
\right)
\end{equation*}
Consider the eigenvalue problem
\begin{equation}\label{theeigenvalueequation}
   D(A+K+\Gamma)u=\lambda u.
\end{equation}
If $\gamma=0$ then the eigenvalue of this problem lie in the half-plane $\Re\lambda<0$. If we show that the are no eigenvalues of (\ref{theeigenvalueequation}) with $\lambda=i\tau$, $\tau\in\Bbb R$, for all $\gamma$ and $K$ satisfying $K\bar{\gamma}\leq \omega$ then by continuity argument for eigenvalues we obtain the required result. Therefore let us assume that one of eigenvalues has the form $\lambda=i\tau,\;\tau\in\mathbb{R}$ and that no eigenvalue has positive real part. We will now show that the problem (\ref{theeigenvalueequation}) has only the trivial solution. For large $K$ and small $\gamma$ (say $K\geq K_*$ and $\bar{\gamma}\leq\bar{\gamma_*}$) we have
\begin{equation}\label{M16aa}
G_8(K)=\Big(\frac{\Delta_\mu}{\Delta_\alpha},\frac{r}{\Delta_\alpha}(A_2-\eta_2),
  \frac{r}{\Delta_\alpha}(\eta_1-A_1),\frac{r}{\Delta_\alpha}A_3\Big)+O(\bar{\gamma} +K^{-1}).
\end{equation}
Therefore
\begin{equation}\label{M16a}
c\leq S\leq C,\,\;c\leq I_1\leq C,\;\;c\leq I_2\leq C,\;\;c\leq I_{12}\leq C,
\end{equation}
where $C$ and $c$ are positive constants depending on $\alpha$ and $\eta$.  Furthermore, the Jacobian matrix at the point
(\ref{M16aa}) is
$$
D\left(\begin{matrix}0&-\alpha_1&-\alpha_2&-\alpha_3\\
\alpha_1&0&0&-\eta_1\\
\alpha_2&0&0&-\eta_2\\
\alpha_3&\eta_1&\eta_2&0
\end{matrix}\right)+O(\bar{\gamma}+K^{-1})
$$
and therefore,
$$
\det J_8(K)=SI_1I_2I_{12}\Delta_\alpha^2+O(\bar{\gamma}+K^{-1}).
$$
Thus we may assume that $\det J_8(K)\geq c_1>0$. This fact together with (\ref{M16a}) gives
$$
0<c_2\leq| \lambda_j(K)\textcolor{black}{|\leq c_3},\;\;j=1,2,3,4,\;\;K\geq K_*\;\;\textcolor{black}{\bar \gamma \leq \bar \gamma_*},
$$
where $\lambda_j(K)$, $j=1,2,3,4$, are eigenvalues of $J_8(K)$.

%

Assume that $\lambda=i\tau,$ $c_2\leq \tau\leq c_3$ is an eigenvalue to $J_8$.
We will now show that this leads to a contradiction. Multiplying both sides of (\ref{theeigenvalueequation}) by $D^{-1}\bar u$ and taking the real part and using that $\Re(Au,u)=0$ we get $=\Re ((K+\Gamma)u,u)=0$
or
\begin{equation}\label{realpartof(K+gamma)u,uin complicated form}
    -\frac{r}{K}|u_1|^2-\bar \gamma r_1r_2|u_4|^2 -\gamma_1\Re(u_3\bar u_2)-\gamma_2 \Re(u_2\bar u_3) + \bar \gamma\Re\{(r_2u_2+r_1u_3)\bar u_4\}=0
 \end{equation}
Let us derive some relations between $u_1,u_2,u_3$ and $u_4$. From the first three equations in (\ref{theeigenvalueequation}) we obtain
\begin{align}\label{firstthreeequationsofsimpleeigenvalueequation}
    S(-\frac{r}{K}u_1-\alpha_1u_2-\alpha_2u_3-\alpha_3u_4)&=i\tau u_1
    \\
    I_1(\alpha_1u_1-\gamma_1u_3-\eta_1u_4)&=i\tau u_2\nonumber
    \\
    I_2(\alpha_2u_1-\gamma_2u_2-\eta_2u_4)&=i\tau u_3\nonumber
\end{align}
We rewrite the last two equations as
\begin{align*}
i\tau u_2+\gamma_1I_1u_3&=\alpha_1I_1u_1-\eta_1I_1u_4
\\
i\tau u_3+\gamma_2I_2u_2&=\alpha_2I_2u_1-\eta_2I_2u_4
\end{align*}
and solving them we obtain
\begin{eqnarray}\label{u_2expressedinu_u_4 and u_1}
    &&u_2=\frac{(-i\tau \alpha_1I_1+\alpha_2\gamma_1I_1I_2)u_1-(\eta_2\gamma_1I_1I_2-i\tau\eta_1I_1)u_4}{\tau^2+\gamma_1\gamma_2I_1I_2}\\
   && u_3=\frac{(\alpha_1\gamma_2I_1I_2-i\tau\alpha_2I_2)u_1    -(\eta_1\gamma_2I_1I_2-i\tau\eta_2I_2)u_4}{\tau^2+\gamma_1\gamma_2I_1I_2}.\label{u_3 expressed in u_1 and u_4}
\end{eqnarray}
Inserting these relations in (\ref{firstthreeequationsofsimpleeigenvalueequation}) we get
\begin{align*}
    &u_1(\frac{i\tau}{S}+\frac{r}{K}+\alpha_1\frac{-i\tau\alpha_1I_1+\alpha_2\gamma_1I_1I_2}{\tau^2+\gamma_1\gamma_2I_1I_2}+ \alpha_2\frac{\alpha_1\gamma_2I_1I_2-i\tau\alpha_2I_2}{\tau^2+\gamma_1\gamma_2I_1I_2})
    \\
    &=u_4(-\alpha_3+\alpha_1\frac{-i\tau\eta_1I_1+\eta_2\gamma_1I_1I_2}{\tau^2+\gamma_1\gamma_2I_1I_2}+\alpha_2\frac{\eta_1\gamma_2I_1I_2-i\tau\eta_2I_2}{\tau^2+\gamma_1\gamma_2I_1I_2})
\end{align*}
This leads to
\begin{equation}\label{|u_4|<c|u_1|}
    |u_4|\leq C_3|u_1|
\end{equation}
The relations (\ref{u_2expressedinu_u_4 and u_1}) and (\ref{u_3 expressed in u_1 and u_4}) together with (\ref{|u_4|<c|u_1|}) gives
$$
    |u_2|,\;|u_3|\leq C_4|u_1|
$$
Now (\ref{realpartof(K+gamma)u,uin complicated form}) implies that
$$
    -\frac{r}{K}|u_1|^2+C_1\bar{\gamma}|u_1|^2=0.
$$
This is impossible if
$C_1$ is sufficiently small. Thus the local stability of $G_8(K)$, $\textcolor{black}{K}\geq K_*$, is proved.

The local stability of $G_8(K)$ for $K\in (\hat{K}_1,K_*]$ is proved in the same manner as in the proof of Proposition \ref{pro:eta_2-1}.

\end{proof}

\section{Equilibrium transition with increasing $K$}

In this section we finalize our results in two theorems describing the equilibrium branch for the sets of parameter $\eta_1^*>\eta^*_2>1$ and $\eta_1^*>1>\eta^*_2$.
\subsection{Equilibrium transition when $\eta_1^*>\eta_2^*>1$}
In this section we will prove that there exist an equilibrium branch $G_2\rightarrow G_3 \rightarrow G_6 \rightarrow G_8 \rightarrow G_7\rightarrow G_5$ in the case $\eta_1^*>\eta_2^*>1$.
By Corollary $5$ in \cite{part1} we know that for these parameters there is an equilibrium branch
$$
    G_2\rightarrow G_3 \rightarrow G_6 \rightarrow \ldots
$$
Furthermore from section \eqref{sec:g6}
we know the this branch continues onto $G_8$ at $K=K_1$.
From section \ref{section the equilibrium state G7} and section \ref{section bifurcation of G7} in this paper as well as Theorem~1 from \cite{part1}, we get that there exist an equilibrium branch
$$
\ldots\rightarrow G_8\rightarrow G_7\rightarrow G_5.
$$
One could suspect that these two equilibrium branches are the two parts of a complete equilibrium branch. We shall now prove that indeed that is the case.

\begin{theorem} Let (\ref{May1aa}), (\ref{Ko1}) and Assumption II hold and let  $\eta_1^*>\eta_2^*>1$. Then there exist a unique  branch of equilibrium points $G^*(K)$ parameterised by $K\in (0,\infty)$:
\begin{enumerate}
    \item [$(a)$] for $0< K\leq \sigma_1$ the point $G^*(K)$ is of type $G_2$

    \item[$(b)$] for $\sigma_1<K\leq \frac{\sigma_1\eta_1^*}{\eta^*_1-1}$ the point $G^*(K)$ is of type $G_3$;
    \item[$(c)$] for $\frac{\sigma_1\eta^*_1}{\eta^*_1-1}<K\leq\frac{\hat S_1\eta_1^*}{\eta^*_1-1}$ the point $G^*(K)$ is of type $G_6$;
    \item[$(d)$] for $\frac{\hat S_1\eta_1^*}{\eta_1^*-1}< K<\frac{\hat S_2\eta_2^*}{\eta_2^*-1}$ the point $G^*(K)$ is of type $G_8$;
    \item[$(e)$] for $\frac{\hat S_2\eta_2^*}{\eta_2^*-1}\leq K<\frac{\sigma_3\eta_2^*}{\eta_2^*-1}$ the point $G^*(K)$ is of type $G_7$;
    \item[($f$)] for $K\geq\frac{\sigma_3\eta_2^*}{\eta_2^*-1}$ the point $G^*(K)$ is of type $G_5$.
\end{enumerate}
we display this schematically as (see figure \ref{figure1})
\begin{equation}
    G_2\rightarrow G_3 \rightarrow G_6 \rightarrow G_8 \rightarrow G_7 \rightarrow G_5.
\end{equation}

The point $G^*(K)$ is locally stable whenever it is not a coexistence point. It is also locally stable near the end on the interval $\hat K_1 <K< \hat K_2$ and it is locally stable on the whole interval if $\bar \gamma$ is small
\end{theorem}

\begin{proof}
This theorem follow from Lemma~\ref{lem:contG_8} and Proposition~\ref{pro:eta_2-1} in section~\ref{sec:locstab}

\end{proof}
\begin{figure}[!ht]
\begin{tikzpicture}[scale=1.1]
\draw[->,name path=xaxis] (0,0) -- (10.2,0) node[right] {$K$};
\draw[->,name path=yaxis] (0,0) -- (0,5.2) node[above] {$S^*$};
\draw[very thick,dashed,color=black,name path=plot,domain=0:5.2] plot (1.5*1,\x);
\draw[very thick,dashed,color=black,name path=plot,domain=0:5.2] plot (1.5*1.5,\x);
\draw[very thick,dashed,color=black,name path=plot,domain=0:5.2] plot (1.5*4,\x);
\draw[very thick,dashed,color=black,name path=plot,domain=0:5.2] plot (1.5*4.66,\x);
\draw[very thick,dashed,color=black,name path=plot,domain=0:5.2] plot (1.5*6,\x);
\draw[very thick,dashed,color=black,name path=plot,domain=0:1.5*7] plot (\x,1.5*2.33);
\draw[very thick,dashed,color=black,name path=plot,domain=0:1.5*7] plot (\x,1.5*2.66);
 %
\node[scale=1] (X) at (1.5*1/2,1.5*1.5) { $G_2$};
\node[scale=1] (X) at (1.5*1.25,1.5*1.5) { $G_3$};
\node[scale=1] (X) at (1.5*5.5/2,1.5*1.5) { $G_6$};
\node[scale=1] (X) at (1.5*4.33,1.5*1.5) { $G_8$};
\node[scale=1] (X) at (1.5*5.33,1.5*1.5) { $G_7$};
\node[scale=1] (X) at (1.5*6.5,1.5*1.5) { $G_5$};
\node[scale=1] (X) at (1.5*1,0) { $|$};
\node[below,scale=1] (X) at (1.5*1,0) {$\sigma_1$};
\node[scale=1] (X) at (1.5*1.5,0) { $|$};
\node[below,scale=1] (X) at (1.5*1.5,0) {$\frac{\sigma_1\eta_1^*}{\eta_1^*-1}$};
\node[scale=1] (X) at (1.5*4,0) {$|$};
\node[below,scale=1] (X) at (1.5*4,0) {$\frac{S_1\eta_1^*}{\eta^*_1-1}$};
\node[scale=1] (X) at (1.5*4.66,0) { $|$};
\node[below, scale=1] (X) at (1.5*4.66,0) { $\frac{S_2\eta_2^*}{\eta^*_2-1}$};
\node[scale=1] (X) at (1.5*6,0) { $|$};
\node[below,scale=1] (X) at (1.5*6,0) {$\frac{\sigma_3\eta_2^*}{\eta^*_2-1}$};
\node[left,scale=1] (X) at (0,1.5) {$\sigma_1$};
\node[scale=1] (X) at (0,1.5) { $-$};
\node[left,scale=1] (X) at (0,3) {$\sigma_2$};
\node[scale=1] (X) at (0,3) { $-$};
\node[left,scale=1] (X) at (0,4) {$S_1$};
\node[scale=1] (X) at (0,4) { $-$};
\node[left,scale=1] (X) at (0,3.5) {$S_2$};
\node[scale=1] (X) at (0,3.5) { $-$};
\node[left,scale=1] (X) at (0,4.5) {$\sigma_3$};
\node[scale=1] (X) at (0,4.5) { $-$};
\draw[-,very thick,color=black,name path=plot,domain=0:1*1.5]
 plot(\x,\x);
 \draw[-,very thick,color=black,name path=plot,domain=1.5*1:1.5*1.5]
 plot(\x,1.5);
  \draw[-,very thick,color=black,name path=plot,domain=1.5*1.5:1.5*4]
 plot(\x,{1.5+2/3*(\x-1.5*1.5)});
 \draw[-,very thick,color=black,name path=plot,domain=1.5*4:1.5*4.66]
plot (\x,{0.05*sin(2*3.14/(1.5*0.66)*(\x-1.5*4) r)+(1.5*2.66-1/2*(\x-1.5*4))});
 \draw[-,very thick,color=black,name path=plot,domain=1.5*4.66:1.5*6]
 plot(\x,{1.5*2.33+1/2*(\x-1.5*4.66)});
  \draw[-,very thick,color=black,name path=plot,domain=9:10]
 plot(\x,1.5*3);
\end{tikzpicture}
\caption{\label{figure1}
This graphs gives the idea of how the $S^*$ component of the equilibrium branch changes with $K$ and shows the type of the equilibrium point. The function $S^*(K)$ is a piecewise linear function except in the interval $\frac{S_1\eta^*_1}{\eta^*_1-1}<K<\frac{S_1\eta^*_2}{\eta^*_2-1}$ where it is strictly decreasing. Note that the order of the elements on both axis is correct.
}
\end{figure}
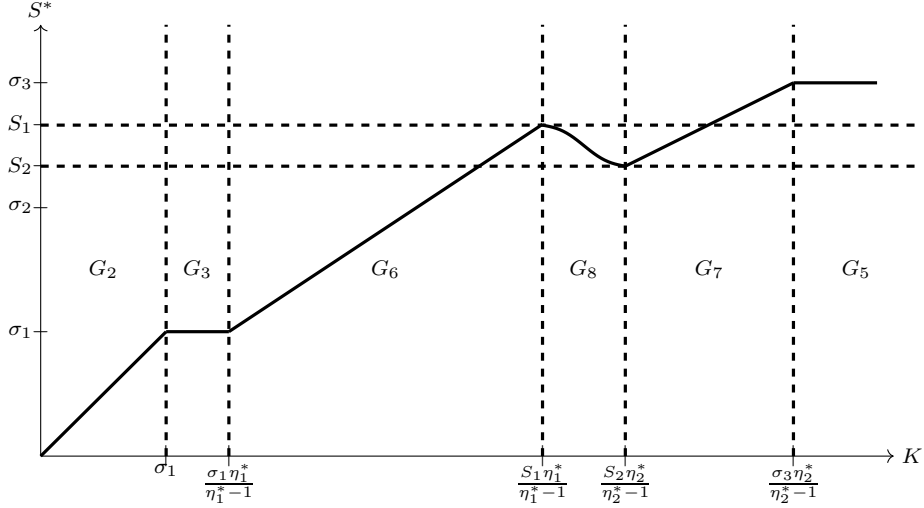
\subsection{Equilibrium transition when $\eta_1^*>1>\eta_2^*$}
 In this section we will prove that there exist an equilibrium branch $\hat G_2\rightarrow \hat G_3 \rightarrow \hat G_6 \rightarrow \hat G_8 $ in the case $\eta_1^*>1>\eta_2^*$.
By Corollary $5$ in \cite{part1} we know that for these parameters there is an equilibrium branch
$$
    \hat G_2\rightarrow \hat G_3 \rightarrow \hat G_6 \rightarrow \ldots
$$
Furthermore from section \ref{sec:g6}
we know the this branch continues onto $G_8$ at $K=\hat K_1$. We are left to prove that this equilibrium does persists.

\begin{theorem} Let (\ref{May1aa}), (\ref{Ko1}) and Assumption II hold and let
 $\eta_1^*>1>\eta_2^*$. Then there exists a unique branch of equilibrium points $G^*(K)$ parameterised by $K\in (0,\infty)$:

\begin{enumerate}
    \item [$(a)$] for $0<K\leq \sigma_1$ the point $G^*(K)$ is of type $G_2$
    \item [$(b)$]  for $\sigma_1<K\leq \frac{\sigma_1\eta_1^*}{\eta^*_1-1}$ the point $G^*(K)$  is of type $G_3$
    \item [$(c)$] for $\frac{\sigma_1\eta^*_1}{\eta^*_1-1}<K\leq\frac{\hat S_1\eta_1^*}{\eta^*_1-1}$ the point $G^*(K)$ is of type $G_6$;
       \item[$(d)$] for $K>\frac{\hat S_1\eta_1^*}{\eta_1^*-1}$ the point $G^*(K)$ is of type $G_8$;
\end{enumerate}

We display this schematically as (see figure \ref{figure2})
\begin{equation}
    \hat G_2\rightarrow \hat G_3 \rightarrow \hat G_6 \rightarrow \hat G_8.
\end{equation}
The point $G^*(K)$ is locally stable whenever it is not a coexistence point. It is also locally stable near the left end on the interval ${\hat K}_1<K<\infty$ and it is locally stable if $K\overline{\gamma}$ is small.

\end{theorem}

%
\begin{proof}
This theorem follow from Lemma  \ref{lem:contG_8} and Theorem \ref{the:Kgamma1} in section \ref{sec: Local stability in the case eta_1^*>1>eta_2^*}.
\end{proof}

\begin{figure}[!ht]
\begin{tikzpicture}[scale=1.1]
\draw[->,name path=xaxis] (0,0) -- (10.2,0) node[right] {$K$};
\draw[->,name path=yaxis] (0,0) -- (0,5.2) node[above] {$S^*$};
\draw[very thick,dashed,color=black,name path=plot,domain=0:5.2] plot (1.5*1,\x);
\draw[very thick,dashed,color=black,name path=plot,domain=0:5.2] plot (1.5*1.5,\x);
\draw[very thick,dashed,color=black,name path=plot,domain=0:5.2] plot (1.5*4,\x);
\draw[very thick,dashed,color=black,name path=plot,domain=0:7] plot (\x,1.5*2.66);
 %
\node[scale=1] (X) at (1.5*1/2,1.5*1.5) { $G_2$};
\node[scale=1] (X) at (1.5*1.25,1.5*1.5) { $G_3$};
\node[scale=1] (X) at (1.5*5.5/2,1.5*1.5) { $G_6$};
\node[scale=1] (X) at (1.5*1,0) { $|$};
\node[below,scale=1] (X) at (1.5*1,0) {$\sigma_1$};
\node[scale=1] (X) at (1.5*1.5,0) { $|$};
\node[below,scale=1] (X) at (1.5*1.5,0) {$\frac{\sigma_1\eta_1^*}{\eta_1^*-1}$};
\node[scale=1] (X) at (1.5*4,0) {$|$};
\node[below,scale=1] (X) at (1.5*4,0) {$\frac{S_1\eta_1^*}{\eta^*_1-1}$};
\node[left,scale=1] (X) at (0,1.5) {$\sigma_1$};
\node[scale=1] (X) at (0,1.5) { $-$};
\node[left,scale=1] (X) at (0,3) {$\sigma_2$};
\node[scale=1] (X) at (0,3) { $-$};
\node[left,scale=1] (X) at (0,4) {$S_1$};
\node[scale=1] (X) at (0,4) { $-$};
\node[left,scale=1] (X) at (0,3.5) {$S_2$};
\node[scale=1] (X) at (0,3.5) { $-$};
\node[left,scale=1] (X) at (0,4.5) {$\sigma_3$};
\node[scale=1] (X) at (0,4.5) { $-$};
\draw[-,very thick,color=black,name path=plot,domain=0:1*1.5]
 plot(\x,\x);
 \draw[-,very thick,color=black,name path=plot,domain=1.5*1:1.5*1.5]
 plot(\x,1.5);
  \draw[-,very thick,color=black,name path=plot,domain=1.5*1.5:1.5*4]
 plot(\x,{1.5+2/3*(\x-1.5*1.5)});
\draw[-,very thick,color=black,name path=plot,domain=1.5*4:1.5*7]
plot(\x,{(2.66*1.5+1.5*2.5*(\x-1.5*4))/(1+\x-1.5*4)});
\end{tikzpicture}
\caption{\label{figure2}
This graphs gives the idea of how the $S^*$ component of the equilibrium branch changes with $K$ and shows the type of the equilibrium point. The function $S^*(K)$ is a piecewise linear function except when $K>\frac{S_1\eta^*_1}{\eta^*_1-1}$ where it is strictly decreasing and converging to a value between $\hat S_1$ and $\hat S_2$. Note that the order of the elements on both axis is correct.
}
\end{figure}
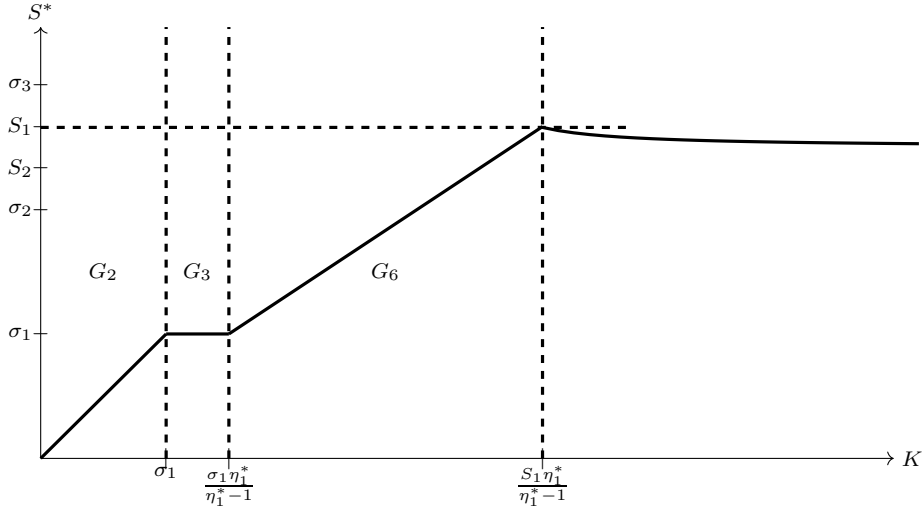

\section{Some concluding remarks}

Below we briefly comment on our results from the biological point of view. We start from $K=0$ and reason how the dynamics changes as $K$ increases.
For small carrying capacity $K$ the susceptible population will be kept so low that the likelihood of an infected individual spreading its \textcolor{black}{disease} will be too low (below 50\% ) for \textcolor{black}{any disease} to spread. As $K$ increases the stable susceptible population increase.

When the stable susceptible population reaches $\sigma_1$, any increase in $S^*$ due to increased $K$ will result in the disease 1 with highest transmission rate to be able to spread. But it can only spread until the susceptible population is equal to $\sigma_1$.
So from now on $S^*=\sigma_1$ and an increases in $K$ gives an increase of $I^*_1$ . \textcolor{black}{Disease 2, with lower transmission rates then disease 1, can not spread since it is outcompeted by disease 1.}

\textcolor{black}{The disease 2 can however spread through the population of infected with disease 1. Under the condition $\frac{\sigma_1\eta_1^*}{\eta_1^*-1}<K<\frac{\min(\sigma_2,\textcolor{black}{\hat S_1})\eta_1^*}{\eta_1^*-1}$ the sum of susceptibles and infected of disease 1 will be so high that disease 2 can spread. However disease 2 will only occur as a coinfection in the stable state. This is a result of the fact that we assume that coinfected individuals can only spread both disease simultaneously. The single infections of disease 2 are either outcompeted by disease 1 or they become part of the coinfected compartment. For these $K$ the compartment of single infected of disease 1 will decrease with $K$. This does however not mean that disease 1 becomes less prevalent, only that it occurs more as a coinfection. The susceptibles increase for these $K$. This is a consequence of the assumption of the coinfection being less transmissible then single infection. When the coinfection rises the average transmission rate of the diseases decrease allowing the susceptible population to increase.}


\textcolor{black}{
For the parameters dealt with in this paper ($\eta^*_1>1$) it will happen that as the average transmission rate of the disease decrease eventually single infection of disease 2 will be more transmissible then disease 1 and the coinfection and will thus be able to spread as a single infection giving rise to a stable coexistence point.
From there either the equilibrium point stays as a coexistence point for all large $K$ or the single infections starts to only occur in coinfections, with disease 1 being the first to stop occurring as a single infection. The susceptible population can only increase to $\sigma_3$ at which point any increase in susceptibles would even be absorbed be the least transmittable compartment (coinfection). The sick population can by assumption not reproduce and so it must also have an upper bound. If this upper bound is large compared to $\sigma_3$ we will have a situation where a large proportion of the population is sick making coinfection far more likely to occur then single infections resulting in the diseases only occuring as coinfection.
while the overall sick population can increase indefinitely. So when $K$ is large enough the number of sick individuals will be far more then the susceptibles making coinfections far more likely to occur then single infection leading to a stable state of coinfection with no single infections.}
%

\appendix

\section{Implicit function theorem}\label{thm:implicitfunction}
Let
$$
{\mathcal F}:\Bbb R^n\times\Bbb R^m\rightarrow\Bbb R^n
$$
be a $C^2$ mapping. Let us consider the equation
\begin{equation}\label{1z}
{\mathcal F}(x,y)=0.
\end{equation}
We assume that
$$
{\mathcal F}(0,0)=0\;\;\mbox{and that the matrix}\;\; A:=D_x{\mathcal F}(0,0)\;\;\mbox{is invertible}.
$$
Our aim is to find a solution to (\ref{1}) $x=x(y)$ such that $x(0)=0$ and estimate the region where such solution exists. We fix positive numbers $a$ and $b$  and put
$$
\Lambda=\Lambda_{a,b}=\{ (x,y)\,:\, |x|\leq a,\;|y|\leq b\}.
$$
Let also
$$
B_a=\{x\,:\, |x|\leq a\}.
$$
We introduce the quantities
$$
M=\max_{\Lambda}||D_xD_x{\mathcal F}(x,y)||,\;\;M_1=\max_{\Lambda}||D_yD_x{\mathcal F}(x,y)||.
$$
Here the above norms are understood in the following sense
$$
||D_xD_x{\mathcal F}(x,y)||=\max_{|\zeta|,|\xi|=1}|\sum_{i,j=1}^n\partial_{x_i}\partial_{x_j}{\mathcal F}(x,y)\zeta_i\xi_j|
$$
Here $|\cdot|$ is the usual euclidian norm.
We also introduce
$$
L=\max_{\Lambda}||D_y{\mathcal F}(x,y)||,
$$
where
$$
||D_y{\mathcal F}(x,y)||=\max_{|\xi|=1}|D_y{\mathcal F}(x,y)\xi|.
$$
The following result is a well known implicit function theorem. We supply it with a short proof since we want to include in the formulation a quantitative information about the solution.

\begin{theorem}
If the constants $a$ and $b$ satisfies
\begin{equation}\label{2}
||A^{-1}||(Ma+M_1b)\leq q\;\;\mbox{and}\;\;||A^{-1}||\Big(Ma+M_1b+L\frac{b}{a}\Big)\leq 1,
\end{equation}
where $q<1$ and $||A^{-1}||$ is the usual operator-norm of $A^{-1}$. 
Then there exist a $C^2$-function $x=x(y)$ defined for $|y|\leq b$  which delivers all solutions to (\ref{1z}) from $\Lambda$.
\end{theorem}
\begin{proof} We write (\ref{1z}) as a fixed point problem
\begin{equation}\label{1a}
x=F(x,y),\;\;\;\mbox{where}\;\;F(x,y)=A^{-1}\big( Ax-{\mathcal F}(x,y)\big).
\end{equation}
Let us check that $F$ maps $B_a$ into itself and that it is a contraction operator there.

To show the first property we note that
$$
{\mathcal F}(x,y)=\int_0^1\frac{d}{dt}{\mathcal F}(tx,ty)dt=\int_0^1\sum_{i=1}^n \partial_{x_i}{\mathcal F}(tx,ty)x_i+\sum_{k=1}^m \partial_{y_k}{\mathcal F}(tx,ty)y_kdt.
$$
Therefore
$$
F(x,y)=A^{-1}\int_0^1\Big(\sum_{i=1}^n(\partial_{x_i}{\mathcal F}(0,0)- \partial_{x_i}{\mathcal F}(tx,ty))x_i-\sum_{k=1}^m\partial_{y_k}{\mathcal F}(tx,ty)y_k\Big)dt.
$$
Since
\begin{eqnarray*}
&&\partial_{x_i}{\mathcal F}(0,0)- \partial_{x_i}{\mathcal F}(tx,ty)=-\int_0^1 \frac{d}{d\tau}(\partial_{x_i}{\mathcal F})(\tau tx,\tau ty)d\tau\\
&&=-\int_0^1\Big( \sum_{j=1}^n\partial_{x_j}\partial_{x_i}{\mathcal F}(\tau tx,\tau ty)tx_j
+\sum_{k=1}^m\partial_{y_k}\partial_{x_i}{\mathcal F}(\tau tx,\tau ty)y_k\Big)d\tau,
\end{eqnarray*}
we get
\begin{eqnarray*}
&&|F(x,y)|\leq|A^{-1}\int_0^1\int_0^1\sum_{i=1}^n\Big(-\sum_{j=1}^n\partial_{x_j}\partial_{x_i}{\mathcal F}(\tau tx,\tau y)tx_j
+\sum_{k=1}^m\partial_{y_k}\partial_{x_i}{\mathcal F}(\tau tx,\tau y)y_k\Big)x_id\tau dt\\
&&+\Big|\int_0^1\sum_{k=1}^m\partial_{y_k}F(tx,ty)y_kdt\Big |
\leq ||A^{-1}||(Ma+M_1b+L\frac{b}{a})|a|<a,
\end{eqnarray*}
which guarantees that $F$ maps $B_a$ on to itself.

For checking the contraction property we write
\begin{eqnarray*}
&&|F(x_1,y)-F(x_2,y)|=|A^{-1}(A(x_1-x_2)-\int_0^1\frac{d}{dt}{\mathcal F}((x_2+t(x_1-x_2),y)dt|
\\
&&\leq ||A^{-1}||\,\Big|\int_0^1\sum_{i=1}^n\partial_{x_i}{\mathcal F}(0,0)- \partial_{x_i}{\mathcal F}(x_2+t(x_1-x_2),y))(x_1-x_2)_idt\Big|
\\
&&\leq||A^{-1}||\int_0^1\Big( \sum_{i,j}|\partial_{x_j}\partial_{x_i}{\mathcal F}(\tau (x_2+t(x_1-x_2),\tau y)tx_j(x_1-x_2)_i|
\\
&&+|\sum_{i,k}\partial_{y_k}\partial_{x_i}{\mathcal F}(\tau (x_2+t(x_1-x_2),\tau y)y_k(x_1-x_2)_i|\Big)d\tau)dt
\leq q|x_1-x_2|,
\end{eqnarray*}
so $F$ is a contraction and by the Banach fixed point theorem we can conclude that there exist a unique $c^1$-function $x=x(y)$ defined for $|y|\leq b$. Since ${\mathcal F}\in C^2$ the same is true for $x(y)$.
\end{proof}

In the next assertion we present estimates of the derivatives of the solution $x(y)$.
\begin{theorem}
The matrix $D_x{\mathcal F}(x,y)$ is invertible for all $(x,y)\in\Lambda$ and
\begin{equation}\label{Apr18a}
|D_x{\mathcal F}(x,y)^{-1}|\leq \frac{||A^{-1}||}{1-q}.
\end{equation}
Furthermore
\begin{equation}\label{Apr18aa}
|D_yx(y)|\leq
\frac{||A^{-1}||L}{1-q},
\end{equation}
and
\begin{equation}\label{Apr18ab}
    ||D_yD_yx||\leq \frac{||A^{-1}||}{1-q}(M\frac{||A^{-1}||^2L^2}{(1-q)^2}+2N\frac{||A^{-1}||L}{1-q}+M_2),
\end{equation}
where
$$
N=\max_{\Lambda}||D_yD_xF(x,y)||,\;\;M_2=\max_{\Lambda}||D_yD_yF(x,y)||.
$$
\end{theorem}
\begin{proof}
With $B=D_x{\mathcal F}(x,y)$ we have $B^{-1}=A^{-1}(I+(B-A)A^{-1})^{-1})$, which gives (\ref{Apr18a}) because
\begin{equation*}
    ||A^{-1}||||B-A||\leq||A^{-1}||(Ma+M_2b)\leq q.
\end{equation*}

Since
\begin{equation}\label{Apr18b}
{\mathcal F}_{x_k}x^k_{y_i}+{\mathcal F}_{y_i}=0,
\end{equation}
we arrive at (\ref{Apr18aa}) by using (\ref{Apr18a}) and definition of $L$.

Derivating once again (\ref{Apr18b}) with respect to $y$ we obtain
with Einsteins summation index
\begin{equation*}
\frac{\partial^2}{\partial y_i\partial y_j}{\mathcal F}={\mathcal F}_{x_kx_l}x^k_{y_i}x^l_{y_j}+{\mathcal F}_{y_ix_l}x^l_{y_j}+{\mathcal F}_{x_py_j}x^p_{y_i}+{\mathcal F}_{x_p}x^{p}_{y_iy_j}+{\mathcal F}_{y_iy_j}=0.
\end{equation*}
Solving for $x_{y_iy_j}$ we get
\begin{eqnarray*}
x_{y_iy_j}=-(F_x)^{-1}(F_{x_kx_l}x^k_{y_i}x^l_{y_j}+F_{y_ix_l}x^l_{y_j}+F_{x_py_j}x^p_{y_i}+F_{y_iy_j}).
\end{eqnarray*}
Using the definitions of norms we obtain (\ref{Apr18ab}).

\end{proof}

\begin{corollary}\label{cor:existenceofsolutionaroundbifurcationpointandestimationofderivatives}
Let $\Lambda_b=\{(x,y):|x|\leq \sqrt{b}, |y|\leq b\}$  and let
\begin{equation}\label{Apr18ca}
\hat M=\sum_{1\leq |\alpha|+k\leq 2}\max_{\Lambda_b}||D^\alpha_{x}D_y^\beta {\mathcal F}(x,y)||.
\end{equation}
If
\begin{equation}\label{Apr18c}
    ||A^{-1}||\hat M (\sqrt{b}+b)\leq c_{n,m},
\end{equation}
where $c_{n,m}$ is a positive constant depending only on $n$ and $m$, then there exist a $C^2$-function $x=x(y)$ defined for $|y|\leq b$ and such that $|x|\leq\sqrt{b}$, which delivers all  solution to \eqref{1z} from $\Lambda_b$. Moreover, the matrix $D_xF(x,y)$ is invertible for all $(x,y)\in \Lambda$ and
\begin{equation*}
|D_xF(x,y)^{-1}|\leq C||A^{-1}||,\;\;
    |D_yx(y)|\leq C||A^{-1}||\hat M,
\end{equation*}
\begin{equation*}
    ||D_yD_yx(y)||\leq C\hat M ||A^{-1}||(1+\hat M||A^{-1}||+\hat M^2||A^{-1}||^2),
\end{equation*}
where $C$  depends only on $n$ and $m$.
\end{corollary}

\section{Bifurcation from a degenerate bifurcation point}\label{section appendix degenerate bifurcation point}
The results of the following section can not be considered as new. They can be deduced from the classical results from \cite{crandall1971bifurcation} and \cite{crandall1973bifurcation}, see also \cite{kielhofer2012introduction} for more complete presentation and related references. Here we give another more direct presentation which are more suitable for application to to models appearing in biological applications. First, systems here are finite dimensional and have a special structure, which essentially simplifies the proofs. Second the bifurcating parameter is fixed from the begining and we are interesting in bifurcation with respect to this parameter. Therefore we present here proofs which are more addapted to our situation.

\subsection{Interior equilibrium point}
Let $x'=(x_1,\ldots,x_{n-1})$ and $x=(x',x_n)$. Consider the problem
\begin{equation}\label{20Jan13a}
{\mathcal F}(x;s)=0,\;\;x\in {\Bbb R}^n,\;\;s\in {\Bbb R},
\end{equation}
 where ${\mathcal F}=({\mathcal F}_1,\ldots,{\mathcal F}_n)^T$. We put $F(x;s)=({\mathcal F}_1,\ldots,{\mathcal F}_{n-1})^T$ and assume that ${\mathcal F}_n(x;s)=f(x;s)x_n$, where $({\mathcal F}_1,\ldots,{\mathcal F}_{n-1})$ and $f$ are real valued function of class $C^2$ with respect to all variables.
Then the problem (\ref{20Jan13a}) can be written as
\begin{equation}\label{1}
F(x;s)=0,
\end{equation}
\begin{equation}\label{2}
x_nf(x;s)=0.
\end{equation}
It is assumed that there exists $x^*\in \Bbb R^{n-1}$ such that
$$
F(x^*,0;0)=0
$$
and that the $(n-1)\times (n-1)$-matrix
\begin{equation}\label{4}
A=\{A_{kj}\}_{k,j=1}^{n-1}=\{\partial_{x_j}{\mathcal F}_k(x^*,0;0)\}_{j,k=1}^{n-1}
\end{equation}
is invertible. This implies, in particular, that the equation
\begin{equation}\label{5}
F(\xi,0;s)=0
\end{equation}
has a solution $\xi(s)\in C^2([-b,b])$  such that $\xi(0)=x^*$. Here $b$ is a  positive number satisfying (\ref{Apr18c}),
where $\hat M$ is given by (\ref{Apr18ca}) with ${\mathcal F}$ replaced by $F$.

This is the only solution to (\ref{5}) in $\Lambda_b$ according to Corallary \ref{cor:existenceofsolutionaroundbifurcationpointandestimationofderivatives}. Moreover this solution is of the class $C^2$ and estimates of the derivatives are presented in the same corollary.
One can easily verify that  $\check{x}(s)=(\xi(s),0)$ solves system (\ref{1}), (\ref{2}) for $s\in [-b,b]$.

We assume that $f(x^*,0;0)=0$ and our goal is to construct a solution to  equations (\ref{1}), (\ref{2}) different from $\check{x}(s)$. This will be achieved if we  solve the problem
\begin{equation}\label{1a}
F(x;s)=0,
\end{equation}
\begin{equation}\label{2a}
f(x;s)=0.
\end{equation}
instead of (\ref{1}), (\ref{2}).
We denote the Jacobian matrix of the  right-hand side at the point $(x^*,0;0)$  by ${\mathcal A}$. Direct calculations show that
$$
{\mathcal A}=\begin{bmatrix}
A&\partial_{x_n}F(x^*,0;0)\\
\nabla_{x'}f(x^*,0;0)&\partial_{x_n}f(x^*,0;0)
\end{bmatrix}.
$$
To find the inverse of the matrix consider the equation
$$
{\mathcal A}(X',X_n)^T=(Y',Y_n).
$$
Then
\begin{equation}\label{15a}
\Theta X_n=\nabla_{x'}f\cdot A^{-1}Y'-Y_n\;\;\mbox{and}\;\;X'=A^{-1}(Y'-\partial_{x_n}F X_n),
\end{equation}
where and in what follows we assume that
$$
\Theta:=\nabla_{x'}fA^{-1}\partial_{x_n}F(x^*,0;0)-\partial_{x_n}f(x^*,0;0)\neq 0.
$$
So the matrix ${\mathcal A}$ is invertible if $A$ is invertible and $\Theta \neq 0$. From (\ref{15a}) it follows the estimates
$$
|X_n|\leq \frac{||A^{-1}||}{|\Theta|}\,|\nabla_{x'}f|\,|Y'|+\frac{1}{|\Theta|}|Y_n|
$$
and
\begin{eqnarray*}
&&|X'|\leq ||A^{-1}||\,(|Y'|+|\partial_{x_n}F|\,|X_n|)\leq ||A^{-1}||\Big(1+\frac{||A^{-1}||\,|\partial_{x_n}F|\,||\nabla_{x'}f||}{|\Theta|}\Big)|Y'|\\
&&+\frac{||A^{-1}||\,|\partial_{x_n}F|}{|\Theta|}Y_n.
\end{eqnarray*}
Therefore
$$
||{\mathcal A}^{-1}||\leq C\Big(||A^{-1}||+\frac{1}{|\Theta|}\big(1+|\partial_{x_n}F|\,||A^{-1}||\big)\big(1+|\nabla_{x'}f|\,||A^{-1}||\big)\Big),
$$
where $C$ is a positive constant depending only on $n$.

Let us introduce the quantity
$$
\mathcal{\hat M}=\sum_{1\leq |\alpha|+k\leq 2}\max_{\Lambda_b}(||D^\alpha_{x}\partial_s^kF(x;s)||+|D^\alpha_{x}\partial_s^kf(x;s)|),
$$
Then according to Corollary \ref{cor:existenceofsolutionaroundbifurcationpointandestimationofderivatives} there exists a solution $\hat{x}(s)=(\hat{x}'(s),\hat{x}_n(s))$ to \eqref{1}-\eqref{2} belonging to $C^2([-b,b])$
  such that $\hat{x}(0)=(x^*,0)$. Here $b$ is a positive number satisfying (\ref{Apr18c}), where $\hat{M}$ is replaced by $\hat{\mathcal M}$.

Let us evaluate the derivative $\frac{d\hat{x}(s)}{ds}$. Differentiating (\ref{1a}) and setting $s=0$, we have
\begin{equation*}
A\frac{d}{ds}\hat{x}'(0)+\partial_{x_n}F(x^*,0;0)\frac{d}{ds}\hat{x}_n(0)+\partial_sF(x^*,0;0)=0.
\end{equation*}
Differentiating $F(\xi(s),0;s)=0$ with respect to $s$ we get $\partial_sF(x^*,0;0)=-A\frac{d}{ds}\xi(0)$. Therefore
\begin{equation}\label{15ba}
A\frac{d}{ds}(x'-\xi)(0)+\partial_{x_n}F(x^*,0;0)\frac{d}{ds}\hat{x}_n(0)=0.
\end{equation}
Writing equation (\ref{2a}) as $f(\hat{x}(s);s)-f(\check{x}(s);s)+f(\check{x}(s);s)=0$ and differentiating it at $s=0$, we get
$$
\nabla_{x'}f(x^*,0;0)\frac{d}{ds}(\hat{x}'-\xi)(0)+\partial_{x_n}f(x^*,0;0)\frac{d}{ds}\hat{x}_n(0)+\frac{d}{ds}f(\check{x};s)\Big|_{s=0}=0.
$$
Therefore
\begin{equation}\label{Apr18d}
\hat{x}_n(s)=\frac{\frac{d}{ds}f(\xi(s),0;s)|_{s=0}}{\Theta}s+O(s^2),
\end{equation}
where $O(s^2)$ is estimated by $Cs^2$, where $C$ depends only on $n$, $||{\mathcal A}^{-1}||$ and $\hat{\mathcal M}$.
Other components are not important for us in this paper so we write only that
$$
\hat{x}'(s)=y^*+O(|s|)
$$
with similar  comment on $O(s)$ as above. From (\ref{15ba}) we can derive a similar formula for the derivative
$$
\frac{d}{ds}\hat{x}'(s)=\frac{d}{ds}\xi(0)-A^{-1}\partial_{x_n}F(x^*,0;0)\frac{d}{ds}\hat{x}_n(0)+O(|s|).
$$

\subsection{On smallest eigenvalue of the Jacobian}
%
The Jacobian matrix for system (\ref{1}), (\ref{2}) is
\[
{\mathcal J}={\mathcal J}(x;s)=\begin{bmatrix}
\partial_{x'}F(x;s)&\partial_{x_n}F(x;s)\\
\partial_{x'}(x_nf)&\partial_{x_n}(x_nf)
\end{bmatrix}.
\]
The Jacobian matrix
$$
{\mathcal J}(x^*,0;0)=\begin{bmatrix}
A &\partial_{x_n}F\\
0&0
\end{bmatrix}\;\;\mbox{at}\;\;(x;s)=(x^*,0;0)
$$
 has a simple eigenvalue $0$. Let us denote the perturbation of this eigenvalue at the point $(x;s)$ by $\lambda=\lambda(x;s)$. The smallest eigenvalue of ${\mathcal J}(\check{x}(s);s)$ is
 \begin{equation}\label{Apr20a}
   \lambda(\check{x}(s);s)=f(\check{x}(s);s)=\frac{d}{ds}f(\xi(s),0;s)|_{s=0}s+O(s^2).
 \end{equation}
   Our aim is to find smallest eigenvalue of ${\mathcal J}(\hat{x}(s);s)$ corresponding to the solution $\hat{x}(s)$.
The eigenvalue equation for the Jacobian at the point $\hat{x}(s)$ is
$$
{\mathcal A}\begin{bmatrix}
X'\\
X_n
\end{bmatrix}=\lambda \begin{bmatrix}
X'\\
X_n
\end{bmatrix}
$$
Without lost of generality we can put $X_n=1$. Solving this system with respect to $X'$ and putting the result in the last equation, we get
$$
-x_n\nabla_{x'}f\cdot\Big(D_{x'}F+x_n\partial_{x'}G(\hat{x};s)-\lambda\Big)^{-1}\partial_{x_n}(F)+\partial_{x_n}(x_nf)=\lambda.
$$
Which implies $\lambda(s)=-\hat{x}(s)\Theta+O(s^2)$ or, using  (\ref{Apr18d}), we get
\begin{equation}\label{Apr18db}
\lambda(\hat{x}(s);s)=-\frac{d}{ds}f(\xi(s),0;s)|_{s=0}s+O(s^2).
\end{equation}
Comparing (\ref{Apr20a}) and (\ref{Apr18db}), we see that the first derivative of smallest eigenvalue corresponding to solutions $\check{x}$ and $\hat{x}$ has opposite sign.

\begin{remark} If we assume that the function $s\to f(\xi(s),0;s)$ is strongly monotone on the interval $[-b,b]$ then all solution to (\ref{1}), (\ref{2}) in the set $|s|\leq b$, $|x'-x^*|^2+x_n^2\leq b$ are exhausted by $\check{x}(s)$ and $\hat{x}(s)$, where $b$ corresponds to $\mu:=\max(||A^{-1}||\hat{M},||{\mathcal A}^{-1}||\hat{\mathcal M})$ in (\ref{Apr18c}). Moreover derivatives of first and second order of this solutions are estimated by constant depending on $\mu$ and $n$ only.

\end{remark}

\bigskip
\noindent {\bf Acknowledgements.} Vladimir Kozlov was supported by the Swedish Research Council (VR), 2017-03837.
%
 \section*{Data availability statement}
The manuscript has no associated data.

 \section*{Compliance with ethical standards}

\textbf{Conflict of interest:} The authors declare that they have no conflict of interests.

\bibliographystyle{plain}%
\nocite{ghersheen2019mathematical}

\end{document}